\documentclass[twoside]{amsart}

\usepackage{amsfonts}

\title[Deformed preprojective algebras]{Deformed preprojective
algebras of
type $L$: K\"ulshammer spaces and derived equivalences}

\author{Thorsten Holm \and Alexander Zimmermann}

\address{~~\newline
Thorsten Holm\newline Institut f\"ur Algebra, Zahlentheorie
und Diskrete Mathematik,
Leibniz Universit\"at Hannover,
Welfengarten 1,
30167 Hannover, Germany}
\email{holm@math.uni-hannover.de}
\urladdr{http://www.iazd.uni-hannover.de/\~{ }tholm}

\address{
~~\newline Alexander Zimmermann\newline
Universit\'{e} de
Picardie et LAMFA (UMR 6140 du CNRS),
33 rue St Leu, 80039 Amiens CEDEX 1, France}
\email{alexander.zimmermann@u-picardie.fr}
\urladdr{http://www.mathinfo.u-picardie.fr/alex/azimengl.html}

\thanks{T.H. is supported by the research grant HO 1880/4-1
of the Deutsche Forschungsgemeinschaft (DFG), in the framework of
the Research Priority Program SPP 1388 Representation Theory.}


\date{preliminary version of \today}

\newtheorem{Lemma}{Lemma}[section]

\newtheorem{Theorem}[Lemma]{Theorem}
\newtheorem{Definition}[Lemma]{Definition}
\newtheorem{Proposition}[Lemma]{Proposition}

\newtheorem{Remark}[Lemma]{Remark}

\newtheorem{Example}[Lemma]{Example}

\newcommand{\lra}{\longrightarrow}

\newcommand{\ra}{\rightarrow}
\newcommand{\sdp}{\times\kern-.2em\vrule height1.1ex depth-.05ex}
\newcommand{\epi}{\lra \kern-.8em\ra}

\newcommand{\N}{{\mathbb N}}

\newcommand{\F}{{\mathbb F}}

\newcommand{\ol}{\overline}

\newcommand{\Z}{{\mathbb Z}}

\newcommand{\dickebox}{{\vrule height5pt width5pt depth0pt}}

\def\Hom{\operatorname{Hom}}
\def\End{\operatorname{End}}
\def\Out{\operatorname{Out}}
\def\Aut{\operatorname{Aut}}
\def\Inn{\operatorname{Inn}}

\def\soc{\operatorname{soc}}
\def\rad{\operatorname{rad}}

 \normalbaselines
\begin{document}

\begin{abstract}
Bia\l kowski, Erdmann and Skowro\'nski classified
those indecomposable self-injective
algebras for which the Nakayama shift of every (non-projective)
simple module is isomorphic to its third syzygy.
It turned out that these are precisely the
deformations, in a suitable sense, of preprojective algebras
associated to the simply laced $ADE$ Dynkin diagrams and
of another graph $L_n$, which also occurs in the
Happel-Preiser-Ringel classification
of subadditive but not additive functions.
In this paper we study these deformed preprojective
algebras of type $L_n$ via their K\"ulshammer spaces,
for which we give precise formulae for their dimensions.
These are known to be invariants of the derived module category,
and even invariants under stable
equivalences of Morita type.
As main application of our study of K\"ulshammer spaces
we can distinguish many (but not all) deformations of the preprojective
algebra of type $L_n$ up to stable equivalence of Morita type,
and hence also up to derived equivalence.
\end{abstract}

\maketitle

\section{Introduction}

Preprojective algebras have been introduced by Gelfand and Ponomarev
\cite{GP} and nowadays occur prominently in various areas in mathematics.
For a quiver (i.e. a finite directed graph) $Q$ its preprojective
algebra is defined by the following process: to any arrow $a$ in $Q$
which is not a loop introduce a new arrow $\overline{a}$ in the opposite
direction; for a loop $a$ set $\overline{a}:=a$, leading
to a new quiver $\overline{Q}$. Then the preprojective algebra $P(Q)$
of type $Q$
over a field $K$ is defined by the quiver with relations $K\overline{Q}/I$
where the ideal is generated by the relations, one for each
vertex $v$ in $Q$, of the form
$\sum_{s(a)=v} a\overline{a}$,
where $s(a)$ denotes the starting vertex of the arrow $a$.
Note that the preprojective algebra is independent of the orientation of
the quiver $Q$. The preprojective algebra for a
quiver associated to a tree is known to be finite-dimensional if
and only if the quiver $Q$ is a disjoint union of some orientations of simply laced
$ADE$ Dynkin diagrams.
The finitely generated modules of the
preprojective algebras for $ADE$ Dynkin quivers have remarkable
homological properties. Namely, by a result of Schofield \cite{Schofield}
each non-projective indecomposable
module has $\Omega$-period at most 6, where $\Omega$ denotes
Heller's syzygy operator; for proofs
see \cite{AR}, \cite{Buchweitz}, \cite{ES}.

In an attempt to characterise those selfinjective finite-dimensional
algebras which share these remarkable periodicity properties,
Bia\l kowski, Erdmann and Skowro\'{n}ski introduced in \cite{BES}
deformations of the preprojective algebras of $ADE$ Dynkin quivers
and of an additional graph of type $L_n$ of the
following form
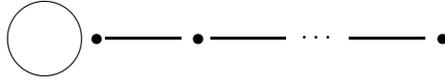
\unitlength 1cm
\begin{center}
\begin{picture}(7,2)
\put(1,1){\circle{1}}
\put(1.6,.9){$\bullet$}
\put(1.8,1){\line(1,0){1}}
\put(2.95,.9){$\bullet$}
\put(3.2,1){\line(1,0){1}}
\put(4.4,1){$\dots$}
\put(5.05,1){\line(1,0){1}}
\put(6.2,.9){$\bullet$}
\end{picture}
\end{center}
which already occurred in the Happel-Preiser-Ringel classification
\cite{HPR} of subadditive but not additive functions.

The deformations $P^f(Q)$ are obtained by perturbing the usual preprojective
relation $\sum_{s(a)=v} a\overline{a}$ at one particular vertex
by adding a certain polynomial expression $f$.
It turns out that proper deformations occur only for the
diagrams of types $D$, $E$, and $L$. For more details on the actual
relations we refer to \cite[Section 3]{BES}.
The Bia\l kowski-Erdmann-Skowro\'{n}ski deformations are different and
should not be confused with the deformed preprojective algebras
of Crawley-Boevey and Holland \cite{CBH}.

The main result of Bia\l kowski, Erdmann and Skowro\'{n}ski gives
the following surprising classification of selfinjective algebras
sharing the periodicity properties of preprojective algebras of Dynkin
type.
\medskip

\noindent
{\bf Theorem.} (\cite[Theorem 1.2]{BES})
{\em Let $\Lambda$ be a basic, connected, finite-dimensional, selfinjective
algebra over an algebraically closed field. Then the following statements
are equivalent:

(i) $\Lambda$ is isomorphic to a deformed preprojective algebra $P^f(Q)$
for a quiver of type $ADE$ or $L$.

(ii) $\Omega^3(S)\cong \nu^{-1}S$ for every non-projective simple
right $\Lambda$-module $S$, where $\nu$ is the Nakayama transformation.
}
\bigskip

In our present paper we shall
study the deformed preprojective algebras of type $L_n$ in
the Bia\l kowski-Erdmann-Skowro\'nski sense.
Let us start by giving a precise
definition of these algebras.

Let $K$ be a field, let $p(X)\in K[X]$ be a polynomial and let $n\in\N$.
Then let $L_n^p$ be the $K$-algebra given by the
following quiver with $n$ vertices
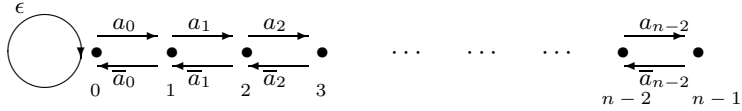
\unitlength1cm
\begin{center}
\begin{picture}(10,2)
\put(1,1){$\bullet$}
\put(1,0.5){\scriptsize $0$}
\put(2,1){$\bullet$}
\put(2,0.5){\scriptsize $1$}
\put(3,1){$\bullet$}
\put(3,0.5){\scriptsize $2$}
\put(4,1){$\bullet$}
\put(4,0.5){\scriptsize $3$}
\put(8,1){$\bullet$}
\put(7.8,0.4){\scriptsize $n-2$}
\put(9,1){$\bullet$}
\put(9,0.4){\scriptsize $n-1$}
\put(5,1){$\cdots$}
\put(6,1){$\cdots$}
\put(7,1){$\cdots$}
\put(1.1,1.3){\vector(1,0){.8}}
\put(2.1,1.3){\vector(1,0){.8}}
\put(3.1,1.3){\vector(1,0){.8}}
\put(8.1,1.3){\vector(1,0){.8}}
\put(8.9,.9){\vector(-1,0){.8}}
\put(3.9,.9){\vector(-1,0){.8}}
\put(2.9,.9){\vector(-1,0){.8}}
\put(1.9,.9){\vector(-1,0){.8}}
\put(0.4,1.1){\circle{1}}
\put(.88,1.1){\vector(0,-1){.1}}
\put(0,1.6){\small $\epsilon$}
\put(1.3,1.4){\small $a_0$}
\put(2.3,1.4){\small $a_1$}
\put(3.3,1.4){\small $a_2$}
\put(8.3,1.4){\small $a_{n-2}$}
\put(1.3,.65){\small $\ol a_0$}
\put(2.3,.65){\small $\ol a_1$}
\put(3.3,.65){\small $\ol a_2$}
\put(8.3,.65){\small $\ol a_{n-2}$}
\end{picture}
\end{center}
subject to the following relations
$$a_i\overline a_i+\overline a_{i-1}a_{i-1}=0\mbox{ for all }
i\in\{1,\dots ,n-2\}~,
$$
$$\overline a_{n-2}a_{n-2} =0~,~~~~~
\epsilon^{2n}=0~,~~~~~
\epsilon^2+a_0\overline a_0+\epsilon^3p(\epsilon)=0.
$$
These algebras are the
{\em deformed preprojective algebras of
type $L_n$}, in the sense of Bia\l kowski,
Erdmann and Skowro\'{n}ski \cite{BES}. Note that the usual
preprojective relations are deformed only at the vertex 0.

More details on these algebras are
collected in Section
\ref{typeLalgebras} below.
In particular we determine their Cartan
matrices and we provide an explicit $K$-basis of the algebra
given by a set of paths in
the quiver. Moreover, we determine explicitly
the centre and the commutator subspace of the deformed preprojective
algebras of type $L_n$.

An important structural property is that all deformed preprojective
algebras of type $L_n$ are symmetric algebras. This is a yet unpublished
result of Bia\l kowski, Erdmann and Skowro\'{n}ski \cite{BESIII};
since we build on it in the present paper we give
an independent proof of this fact for the sake of completeness
(cf. Section \ref{Sec:symmetric}).

It is a subtle question for which deformation polynomials $p$ the deformed
preprojective algebras $L_n^p$ become isomorphic.
We have been informed by Skowro\'{n}ski (cf. also the talk of Bia\l kowski in 
Tokyo at the ICRA XIV) that
over a field $K$ of characteristic different from $2$ all deformed
preprojective algebras of type $L_n$ are isomorphic. However, in
characteristic $2$ the situation is more complex; Bia\l kowski, Erdmann
and Skowro\'{n}ski have given a series of pairwise non-isomorphic
deformed preprojective algebras of type $L_n$ (over an algebraically
closed field), see \cite[Proposition 6.1]{BES}.
Moreover, they have even announced \cite{Bialkowski-ICRA-Abstract} a complete classification
of the deformed preprojective algebras of type $L_n$ up to Morita equivalence
(over an algebraically closed field of characteristic $2$); namely,
the algebras corresponding to the set of deformation
polynomials $p(X)=X^{2j}$ for $j\in\{0,1,2,\dots,n-1\}$ give
a complete classification up to Morita equivalence.

We are not building on this classification in the present paper but we
use it as a motivation for restricting our computations of K\"ulshammer
spaces in Section \ref{mainresultsection} to the case of deformation polynomials
$X^{2j}$.

The Bia\l kowski-Erdmann-Skowro\'{n}ski characterisation of the selfinjective
algebras where for each non-projective simple module
the third syzygy is isomorphic to the
(inverse) Nakayama transformation can be seen as a
condition on the stable module category; we therefore believe it is
natural to aim at a classification of the deformed preprojective algebras
up to stable equivalence or up to derived equivalence,
rather than up to Morita equivalence.

Our main results in this paper provide partial answers to these problems.
We are able to distinguish several of the deformed preprojective algebras
$L_n^{X^{2j}}$ up to stable equivalence of Morita type and up to derived
equivalence.
Our main applications in this direction are summarised in the following
result.

\begin{Theorem}
Let $K$ be a perfect field of characteristic 2.
\begin{enumerate}
\item[{(a)}]
If two deformed preprojective
algebras $L_n^{p}$ and $L_m^{q}$ are stably equivalent of Morita type
or derived equivalent,
then $n=m$.
\item[{(b)}]
For $n\in\mathbb{N}$ let $j,k\in\{0,1\ldots,n-1\}$ be different numbers
such that
$\{j,k\}\neq \{n-2r,n-2r-1\}$ for every $1\le r\le \lceil \frac{n-2}{2}\rceil$.
Then the deformed preprojective algebras
$L_n^{X^{2j}}$ and $L_n^{X^{2k}}$
are not stably equivalent of Morita type, and also not
derived equivalent.
\end{enumerate}
\end{Theorem}

These results are obtained as a consequence of a detailed study
of the K\"ulshammer spaces for the deformed preprojective algebras
of type $L$. These spaces have been defined by K\"ulshammer in the
1980's for symmetric algebras over a field of positive characteristic;
we recall briefly the construction and some fundamental properties
from \cite{Ku1}.
For an algebra $A$ over a field $K$ let
$[A,A]$ be the $K$-vector space generated by
$\{ab-ba\in A\;|\;a,b\in A\}$ and call this space the commutator subspace of $A$.
K\"ulshammer defined  for a $K$-algebra $A$ over a perfect field $K$
of characteristic $p>0$ the $K$-vector spaces
$T_i(A):=\{x\in A\;|\;x^{p^i}\in [A,A]\}$ for every integer $i\ge 0$.
They form an ascending series
$$[A,A]=T_0(A) \subseteq T_1(A) \subseteq T_2(A) \subseteq \ldots \subseteq
T_i(A) \subseteq T_{i+1}(A) \subseteq \ldots.$$
In \cite{Z2} it was shown by the second author that for symmetric algebras
over a perfect field
the codimension of the commutator space of
$A$ in $T_i(A)$ is invariant under derived equivalences, and in \cite{LZZ} Liu, Zhou
and the second author showed that this codimension is an invariant under stable
equivalences of Morita type. In joint work with Bessenrodt \cite{BHZ} we showed that
the codimension of $T_i(A)$ in $A$ is an invariant of the derived category
of $A$ for general (not necessarily symmetric) finite dimensional algebras.

The derived invariance of the various codimensions of
K\"ulshammer spaces proved already to be
very useful to distinguish derived equivalence classes
of symmetric algebras, see \cite{HS}, \cite{HS1}, \cite{hz-tame},
and also to distinguish
stable equivalence classes of Morita type, see \cite{ZZ1}, \cite{ZZ2}.

For obtaining our above results on deformed preprojective algebras of type $L_n$
(over a perfect field of characteristic $2$) we determine the dimension of
their K\"ulshammer spaces $T_i(L_n^{X^{2j}})$;
our main result in this direction is the following.

\begin{Theorem} \label{mainresultintro} Let $K$ be a perfect field of
characteristic $2$. Then for every $0\leq j< n$ we have
\begin{enumerate}
\item[{(a)}]
$\dim_K T_i(L_n^{X^{2j}})-\dim_K [L_n^{X^{2j}},L_n^{X^{2j}}]
= n-\max\left(\left\lceil\frac{2n-(2^{i+1}-2)j-(2^{i+1}-1)}{2^{i+1}}
\right\rceil,0\right)$
\item[{(b)}]
$\dim_K [L_n^{X^{2j}},L_n^{X^{2j}}] = \frac{1}{3}n(n-1)(2n+5).$
\end{enumerate}
\end{Theorem}

\medskip

The paper is organised as follows. In Section~\ref{Nakayamatwistedsection}
we recall some results for selfinjective algebras and we
propose a method to compute the centre and the quotient of the algebra modulo
the commutator space for selfinjective algebras which we believe should be useful
in other situations as well.
In Section~\ref{typeLalgebras} we study the deformed preprojective
algebras $L_n^p$, give a $K$-basis, the Cartan matrix, the commutator space,
and the centre of the algebras. In Section~\ref{mainresultsection}
we compute the K\"ulshammer spaces $T_n(L_n^{X^{2j}})$ and deduce the main results.

\section{Hochschild homology and Nakayama automorphisms
of self-injective algebras}

\label{Nakayamatwistedsection}

In this section we present some general methods to deal with
selfinjective algebras, in particular for getting a Nakayama
automorphism and related bilinear forms explicitly.
Strictly speaking the results of this section are not used
in this generality in the rest of this paper since
the deformed preprojective algebras of type $L$ are symmetric
(we give an independent proof of this result of Bia\l kowski,
Erdmann and Skowro\'{n}ski in Section \ref{Sec:symmetric}
below).
However, symmetricity of an algebra is usually not easy to verify
so that the methods of this section can be used to deal with
K\"ulshammer ideals in cases where
one only has selfinjectivity; therefore the methods of
this section might be of independent interest.

We need to compute rather explicitly in the degree $0$ Hochschild homology of
self-injective algebras. This needs some theoretical preparations in order
to be able to determine a basis of the commutator subspace of the algebras
we need to deal with.

\subsection{The Nakayama-twisted centre}\label{Nakayamatwistedcentre}

Let $K$ be a field and let $A$ be a $K$-algebra.
We need to get alternative descriptions of the degree $0$
Hochschild homology.
By definition of Hochschild homology (using the standard Hochschild
complex) we have $HH_0(A)\cong A/[A,A]$.

If $A$ is symmetric, then by definition
$A\simeq \Hom_K(A,K)$ as $A$-$A$-bimodules (i.e. as $A\otimes_K A^{op}$-module),
and so we get
\begin{eqnarray*}
\Hom_K (A/[A,A],K) & \simeq &
\Hom_K(HH_0(A),K) \simeq \Hom_K(A\otimes_{A\otimes_K A^{op}}A,K) \\
& \simeq &  \Hom_{A\otimes_KA^{op}}(A,\Hom_K(A,K))
\simeq
\Hom_{A\otimes_K A^{op}}(A,A)\\
& \simeq &  HH^0(A)
\simeq  Z(A).
\end{eqnarray*}
This chain of isomorphisms is one of the main tools for the proof
of the main theorem in \cite{Z}.

If $A$ is only self-injective we shall give an analogous description.
So we need to understand $\Hom_K(A,K)$ as $A\otimes_KA^{op}$-module.
If $A$ is a self-injective $K$-algebra then still
$A\simeq \Hom_K(A,K)$ as a left $A$-module. Hence, $\Hom_K(A,K)$
is a free left $A$-module of rank $1$. Moreover,
$$\End_A(\ _A\Hom_K(A,K))\simeq \End_A(\ _AA)\simeq A$$
and so $\Hom_K(A,K)$ is a progenerator over $A$ with endomorphism
ring isomorphic to $A$, hence inducing a Morita self-equivalence of $A$.
Therefore the isomorphism class of
$\Hom_K(A,K)$ is in the Picard group $Pic_K(A)$.
Moreover, as $\Hom_K(A,K)$ is free of rank $1$
as left-module one gets that $\Hom_K(A,K)$ is in the image of
$$\Out_K(A):=\Aut_K(A)/\Inn(A)$$
in $Pic_K(A)$,
where this identification is given by sending $\alpha\in \Aut_K(A)$
to the invertible bimodule  $_1A_\alpha$
which is $A$ as vector space, on which $a\in A$ acts by multiplication
on the left and by $\alpha(a)$ on the right (cf. \cite[(37.16) Theorem]{MO}).
Hence, there is an automorphism $\nu\in \Aut_K(A)$ so that
$$\Hom_K(A,K)\simeq \ _1A_\nu$$
as $A$-$A$-bimodules and $\nu$ is unique up to an inner automorphism.

\begin{Definition}
Let $A$ be a self-injective $K$-algebra. Then there is an automorphism
$\nu$ of $A$ so that $\Hom_K(A,K)\simeq \ _1A_\nu$ as $A$-$A$-bimodules.
This automorphism is unique up to inner automorphisms and is called the
{\em Nakayama automorphism}.
\end{Definition}

For the dual of the degree $0$ Hochschild homology we get
\begin{eqnarray*}
\Hom_K(A/[A,A],K) & \simeq &
\Hom_K(HH_0(A),K)
\simeq  \Hom_K(A\otimes_{A\otimes_K A^{op}} A,K) \\
& \simeq & \Hom_{A\otimes_K A^{op}}(A,\Hom_K(A,K))
 \simeq  \Hom_{A\otimes_K A^{op}}(A,\,_1A_{\nu}) \\
& \simeq & \{a\in A\;|\;b\cdot a=a\cdot\nu(b)\mbox{~for all~} b\in A\}
\end{eqnarray*}
where the last isomorphism is given by sending a homomorphism to
the image of $1\in A$.

\begin{Definition} \label{Def-twistedcentre}
Let $A$ be a self-injective $K$-algebra with
Nakayama automorphism $\nu$. Then the {\em Nakayama twisted centre} is
defined to be
$$Z_\nu(A):=\{a\in A\;|\;b\cdot a=a\cdot \nu(b)\mbox{~for all~} b\in A\}.$$
\end{Definition}

\begin{Remark} \label{Rem-Nakayama}
{\em (1) The automorphism $\nu$ is unique only up to an inner automorphism.
If $\nu$ is inner, let $\nu(a)=u\cdot a\cdot u^{-1}$. Then
\begin{eqnarray*}
\{a\in A\;|\;b\cdot a=a\cdot\nu(b)\mbox{~for all~} b\in A\}
&=&\{a\in A\;|\;b\cdot a=a\cdot u\cdot b\cdot u^{-1}\mbox{~for all~} b\in A\}\\
&=&\{a\in A\;|\;b\cdot (a\cdot u)=(a\cdot u)\cdot b\mbox{~for all~} b\in A\}\\
&=&\{a\in A\;|\;a\cdot u\in Z(A)\}
= Z(A)\cdot u^{-1}
\end{eqnarray*}
and likewise the twisted centres with respect of two different Nakayama
automorphisms differ by multiplication by a unit.

(2) In general the Nakayama twisted centre will not be a ring.
However, if $z\in Z(A)$ and $a\in Z_\nu(A)$ then
$$b\cdot za=zba=za\cdot\nu(b)$$
and $za\in Z_\nu(A)$. Hence $Z_\nu(A)$
is a $Z(A)$-submodule of $A$. The module structure does not depend on the
chosen Nakayama automorphism, up to isomorphism of $Z(A)$-modules.
}
\end{Remark}

We summarise the above discussion in the following Lemma.

\begin{Lemma}\label{propertiesofNakayamatwistedcentre}
If $A$ is a self-injective $K$-algebra, then there is
an automorphism
$\nu$ of $A$, unique up to an inner automorphism so that
$\Hom_K(A,K)\simeq\ _1A_\nu$ as an $A$-$A$-bimodule and
$\Hom_K(HH_0(A),K)\simeq Z_\nu(A)$
as $Z(A)$-modules.

The selfinjective algebra $A$ is symmetric if and only if the Nakayama automorphism $\nu$
is inner. \dickebox
\end{Lemma}

\begin{Remark} \label{Rem-dimension-Nakayamatwisted}
{\em (1) The automorphism $\nu$ is the well-known
Nakayama automorphism. (The diligent reader might observe that
we are dealing with left modules while originally
Nakayama in \cite{Nakayama2}
dealt with right modules, so our $\nu$ would be the inverse
of the original Nakayama automorphism.)

(2) Using that $HH_0(A)\cong A/[A,A]$, the dimension of the
commutator subspace of a selfinjective algebra $A$ can therefore
be expressed as
$$\dim_K [A,A] = \dim_K A - \dim_K Z_\nu(A).$$
}
\end{Remark}

\subsection{How to get the Nakayama automorphism explicitly}

Let $K$ be a field and let $A$ be a self-injective $K$-algebra.
In order to compute the Nakayama automorphism $\nu$ we need to find an explicit
isomorphism $A\lra \Hom_K(A,K)$ as $A$-modules. Suppose we get two isomorphisms
$\alpha_1:A\lra \Hom_K(A,K)$ and $\alpha_2:A\lra \Hom_K(A,K)$. Then
$\alpha_2^{-1}\circ \alpha_1:A\lra A$ is an automorphism of the regular
$A$-module $A$. Hence, $\alpha_1$ will differ from $\alpha_2$ by
multiplication by an invertible element $u\in A$.
The corresponding Nakayama automorphisms $\nu_1$ and $\nu_2$ computed from
$\alpha_1$ and from $\alpha_2$ will then differ by the inner automorphism
given by conjugation with $u$. It is therefore sufficient to find one
isomorphism $\alpha:A\lra \Hom_K(A,K)$. Given such an isomorphism $\alpha$
of $A$-modules, the form $\langle x,y\rangle_\alpha:=(\alpha(y))(x)$ for
$x,y\in A$ is a non degenerate associative bilinear form on $A$.

Let $\langle\;,\;\rangle:A\times A\lra K$ be a non-degenerate
associative $K$-bilinear form on $A$ (which exists since $A$ is
self-injective), then we get a vector space isomorphism
$$A \, \stackrel{\alpha}{\lra} \, \Hom_K(A,K)~,~~~
a \, \mapsto \, \langle -,a \rangle.
$$

\begin{Lemma}\label{explicitselfinjective}
A non-degenerate associative bilinear form
$\langle\;,\;\rangle:A\times A\lra K$
induces an isomorphism
$A\lra \Hom_K(A,K)$ as left
$A$-modules by mapping $a\in A$ to the linear form
$$A\ni b\mapsto \langle b,a\rangle\in K.$$
\end{Lemma}

\begin{proof} By the above discussions the map is
an isomorphism of vector spaces.
For verifying the module homomorphism property
recall the action of $A$ on the dual space
$\Hom_K(A,K)$; it is given by
$(b\cdot \varphi)(c)= \varphi(cb)$ for all $b,c\in A$
and all $\varphi\in \Hom_K(A,K)$.
Then, using that the bilinear form is associative we get
$$\alpha(b\cdot a)(c)=\langle c,b\cdot a\rangle=
\langle c\cdot b,a\rangle=\left(b\cdot\alpha(a)\right)(c)$$
for all $a,b,c\in A$, so the map is a homomorphism
of left $A$-modules.
\end{proof}

\begin{Proposition}\label{thenuformula}
Let $K$ be a field and let $A$ be a self-injective $K$-algebra.
Then the Nakayama automorphism $\nu$ of $A$
satisfies $\langle a,b\rangle=\langle b,\nu(a)\rangle$ for all $a,b\in A$,
and any automorphism satisfying this formula is a Nakayama automorphism.
\end{Proposition}

\begin{proof}
There is a non-degenerate associative bilinear
form on $A$, which induces an isomorphism between $A$ and the linear forms
on $A$ as $A$-modules by Lemma~\ref{explicitselfinjective}.
The isomorphism gives an isomorphism of $A$-$A$-bimodules of
$_1A_\nu$ and $\Hom_K(A,K)$
by
\begin{eqnarray*}
_1A_\nu&\stackrel{\varphi}{\lra} \Hom_K(A,K)\mbox{~~~,~~~}
a \mapsto \langle -,a\rangle=\varphi(a).
\end{eqnarray*}
By the twisted bimodule action on $_1A_{\nu}$ we have
that $\varphi(1)\cdot a = \varphi(1\cdot a) =
\varphi(\nu(a))$ and
$b\cdot\varphi(1) = \varphi(b\cdot 1) = \varphi(b)$.
Since for $f\in \Hom_K(A,K)$ the $A$-$A$-bimodule action
on $\Hom_K(A,K)$ is given by $(fa)(b)=f(ab)$ and $(af)(b)=f(ba)$
for all $a,b\in A$, one gets
\begin{eqnarray*}
\langle a,b\rangle&=&(\varphi(b))(a)
=(b\cdot \varphi(1))(a) = \varphi(1)(ab)
=(\varphi(1)\cdot a)(b) = \varphi(\nu(a))(b)
 = \langle b,\nu(a)\rangle.
\end{eqnarray*}
Hence, the Nakayama automorphism has the above property.
Conversely, if an automorphism $\nu$ satisfies
$\langle a,b\rangle=\langle b,\nu(a)\rangle$ for all $a,b\in A$,
then the mapping $A\lra \Hom_K(A,K)$ given by $a\mapsto \langle -,a\rangle$
gives an isomorphism
of $A$ and $\Hom_K(A,K)$ as $A$-modules, inducing the element $_1A_\nu$ in the
Picard group of $A$.
\end{proof}

We shall later need such a
bilinear form explicitly.
The following very useful result
can be found in \cite[Proposition 2.15]{Z2}; see also
\cite[Proposition 3.1]{hz-tame} for a proof in the case of
weakly symmetric algebras.

\begin{Proposition}
\label{prop:form}
Let $A=KQ/I$ be a self-injective algebra given by the quiver $Q$ and
ideal of relations $I$, and fix
a $K$-basis ${\mathcal B}$ of $A$ consisting of pairwise
distinct non-zero paths of the quiver $Q$. Assume that
${\mathcal B}$ contains a basis of the socle $soc(A)$ of $A$.
Define a $K$-linear mapping $\psi$ on the basis elements by
$$
\psi(b)=\left\{
\begin{array}{ll} 1 & \mbox{if $b\in soc(A)\setminus\{0\}$} \\
                  0 & \mbox{otherwise}
\end{array} \right.
$$
for $b\in {\mathcal B}$.
Then an associative non-degenerate $K$-bilinear
form $\langle-,-\rangle$ for $A$ is given by
$\langle x,y\rangle := \psi(xy).$
\end{Proposition}

\begin{Remark}
{\em
The above bilinear form is in general not symmetric,
even if the algebra $A$ is symmetric. For explicit examples
we refer to \cite[Section 4, proof of main theorem, part (3)]{HS1}
and \cite{Z2}.
}
\end{Remark}

Actually, this form is basically the only possible form, at least for
finite dimensional basic
selfinjective algebras over an algebraically closed field $K$.

\begin{Proposition}
Let $A$ be a finite dimensional basic selfinjective $K$-algebra over an
algebraically closed field $K$.
Then for every non degenerate associative bilinear form
$\langle.\;.\rangle:A\times A\lra K$
there is a $K$-basis $\mathcal B$ containing a $K$-basis of the socle so that
$\langle x,\;y\rangle=\psi(xy)$ where
$$\psi(b)=\left\{
\begin{array}{ll} 1 & \mbox{if $b\in soc(A)\setminus\{0\} $} \\
                  0 & \mbox{otherwise}
\end{array} \right.
$$
for $b\in {\mathcal B}$.
\end{Proposition}

\begin{proof}
Given an associative bilinear form $\langle.,.\rangle:A\times A\lra K$
there is a linear map $\psi:A\lra K$ defined by $\psi(x):=\langle 1,x\rangle$
and for any $x,y\in A$ one gets $\langle x,y\rangle=\langle 1,xy\rangle=\psi(xy)$.
Hence $\psi$ determines the associative bilinear map and the associative
bilinear map determines $\psi$.

The algebra is basic and so the socle of $A$ is a direct sum of pairwise non-isomorphic
one-dimensional simple $A$-modules. Let $\{s_1,\dots,s_n\}\in A$ so that $s_iA$ is simple
for every $i\in\{1,2,\dots,n\}$ and so that $\soc(A)=<s_1,\dots,s_n>_K$.

Given an associative non degenerate bilinear form $\langle.,.\rangle$
then $\langle .,s_i\rangle:A\lra K$ is a non zero linear form on $A$, since
the bilinear form is non degenerate. Hence there is an element $a\in A$ so that
$\langle a,s_i\rangle\neq 0$. Now, by the Wedderburn-Mal\c cev theorem, there
is an element $\rho\in \rad(A)$ so that $a=\sum_{i=1}^n\lambda_ie_i+\rho$
for scalars $\lambda_i\in K$, and
where $e_i^2=e_i$ is an indecomposable idempotent of $A$, where
$e_{\nu^{-1}(i)}s_i=s_i$, and where $e_{\nu^{-1}(j)}s_i=0$ for $j\neq i$.
Hence,
$$\langle a,s_i\rangle=\langle 1,as_i\rangle=\langle
1,\lambda_{\nu^{-1}(i)}s_i\rangle=\lambda_{\nu^{-1}}(i)$$
We replace $s_i$ by $\lambda_{\nu^{-1}(i)}^{-1}s_i$ and get $\langle 1,s_i\rangle=1$.
Take a $K$-basis
${\mathcal B}_i$ of $\ker(\langle .,s_i\rangle)$ in $Ae_{\nu^{-1}(i)}$. Then,
since $A=\bigoplus_{j=1}^n Ae_j$,
$${\mathcal B}:=\bigcup_{i=1}^n{\mathcal B}_i\cup\{s_1,s_2,\dots,s_n\}$$
is a $K$-basis of $A$ satisfying the hypotheses of Proposition~\ref{prop:form}.
Moreover, if $xy\in {\mathcal B}$, then there is a unique $e_i^2=e_i$ so that
$xye_i=xy$, and so $$\langle x,y\rangle=\langle 1,xy\rangle=
\sum_{i=1}^n \langle 1,xye_i\rangle=
\left\{\begin{array}{ll}1&\mbox{ if }xy\in \soc(A)\\0&\mbox{ else}\end{array}
\right.$$
This shows the statement.
\end{proof}

\section{Deformed preprojective algebras of type $L$}
\label{typeLalgebras}

\subsection{$K$-bases of the deformed preprojective algebras of type $L$}
The aim of this section is to obtain an explicit vector space basis
for any deformed preprojective algebra of type $L$ and to deduce some
structural properties. In particular we shall get the Cartan matrices
and provide an independent proof of a result of Bia\l kowski, Erdmann
and Skowro\'{n}ski \cite{BESIII}
that the deformed preprojective algebras of type $L$ are symmetric
algebras.

For the conveneince of the reader we start by recalling
from the introduction the
definition of the deformed preprojective algebras of type $L$.

Let $K$ be a field.
For any $n\in\N$ and any polynomial $p(X)\in K[X]$
let $L_n^p$ be the $K$-algebra given by the
following quiver with $n$ vertices $0,1,\dots,n-1$ of the form
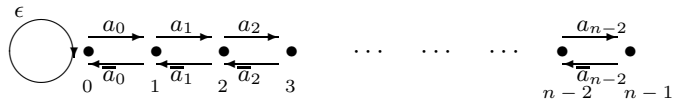
\unitlength0.9cm
\begin{center}
\begin{picture}(10,2)
\put(1,1){$\bullet$}
\put(1,0.5){\scriptsize $0$}
\put(2,1){$\bullet$}
\put(2,0.5){\scriptsize $1$}
\put(3,1){$\bullet$}
\put(3,0.5){\scriptsize $2$}
\put(4,1){$\bullet$}
\put(4,0.5){\scriptsize $3$}
\put(8,1){$\bullet$}
\put(7.8,0.4){\scriptsize $n-2$}
\put(9,1){$\bullet$}
\put(9,0.4){\scriptsize $n-1$}
\put(5,1){$\cdots$}
\put(6,1){$\cdots$}
\put(7,1){$\cdots$}
\put(1.1,1.3){\vector(1,0){.8}}
\put(2.1,1.3){\vector(1,0){.8}}
\put(3.1,1.3){\vector(1,0){.8}}
\put(8.1,1.3){\vector(1,0){.8}}
\put(8.9,.9){\vector(-1,0){.8}}
\put(3.9,.9){\vector(-1,0){.8}}
\put(2.9,.9){\vector(-1,0){.8}}
\put(1.9,.9){\vector(-1,0){.8}}
\put(0.4,1.1){\circle{1}}
\put(.88,1.1){\vector(0,-1){.1}}
\put(0,1.6){\small $\epsilon$}
\put(1.3,1.4){\small $a_0$}
\put(2.3,1.4){\small $a_1$}
\put(3.3,1.4){\small $a_2$}
\put(8.3,1.4){\small $a_{n-2}$}
\put(1.3,.65){\small $\ol a_0$}
\put(2.3,.65){\small $\ol a_1$}
\put(3.3,.65){\small $\ol a_2$}
\put(8.3,.65){\small $\ol a_{n-2}$}
\end{picture}
\end{center}
subject to the following relations
$$a_s\overline a_s+\overline a_{s-1}a_{s-1}=0\mbox{ for all }
s\in\{1,\dots ,n-2\}~,
$$
$$\overline a_{n-2}a_{n-2} =0~,~~~~~
\epsilon^{2n}=0~,~~~~~
\epsilon^2+a_0\overline a_0+\epsilon^3p(\epsilon)=0.
$$
Our first aim is to
give a $K$-basis of the algebra $L_n^p$.
We start by providing a generating set.
Considering a path starting at the vertex $i$ and ending at
the vertex $j$, we have two cases.

Firstly suppose that the path does not contain $\epsilon$.

If $i<j$, then using the relations
$a_s\overline a_s+\overline a_{s-1}a_{s-1}=0\mbox{ for all }
s\in\{1,\dots ,n-2\}$ in $L_n^p$
we may replace the path, up to a sign,
by one of the following elements of $L_n^p$:

$\bullet$ the path $a_ia_{i+1}\dots a_{j-1}$

$\bullet$ or the path $a_ia_{i+1}\dots a_{j-1}a_j\dots a_\ell\ol a_\ell
\ol a_{\ell-1}\dots \ol a_j$ (for some $j\le l\le n-2$)

$\bullet$ or by $0$.

In fact up to a sign we can order the arrows in the path so that
all $a_r$'s come first and then all the $\ol a_r$'s; we can do this
unless we hit a subpath $\ldots \ol a_{n-2}a_{n-2}\ldots$ in which case
the path becomes 0 in $L_n^p$.

Similarly, if $i\geq j$
we may replace the given path, up to a sign, by one of the following elements:

$\bullet$ the path $\ol a_{i-1}\ol a_{i-2}\dots \ol a_{j}$

$\bullet$ or the path
$a_ia_{i+1}\dots a_\ell\ol a_\ell
\ol a_{\ell-1}\dots \ol a_i \ldots \ol a_j$ (for some $i\le l\le n-2$)

$\bullet$ or by $0$.

Secondly, suppose the path contains $\epsilon$.

Using the relations
$\epsilon^2+\epsilon^3p(\epsilon)+a_0\ol a_0=0$ and
$\epsilon^{2n}=0$
we may replace any
path containing powers of $\epsilon$ by a linear combination of paths
containing only $\epsilon$.
Moreover, using the relations
$a_s\overline a_s+\overline a_{s-1}a_{s-1}=0$ for all $s\in\{1,\dots ,n-2\}$
and the fact that $a_0\ol a_0$ commutes with $\epsilon$
(because $a_0\ol a_0=-(\epsilon^2+\epsilon^3p(\epsilon))$)
we can move all $a_r$'s in the path to the right of $\epsilon$.
(Note that by combining these two reductions we can indeed
guarantee that in each path
occurring in the linear combination $\epsilon$ occurs only once.)
Thus
the given path represents the same element in $L_n^p$ as a linear combination
of paths of the following forms

$\bullet~ \ol a_{i-1}\ol a_{i-2}\dots \ol a_{0}\epsilon a_0 a_1\dots a_{j-1}$

$\bullet~\ol a_{i-1}\ol a_{i-2}\dots \ol a_{0}\epsilon a_0 a_1\dots
a_{\ell-1}a_{\ell}\ol a_{\ell}\ol a_{\ell-1}\dots \ol a_{j}
\mbox{~~~(for some $j\le l\le n-2$)}.$

If $i<j$ these paths are all non-zero and they will be part
of the basis to be given below. However, if $i\ge j$ some of the
paths of the latter type vanish, so we shall now derive a different expression
for these.

To this end observe that by using the relations
$a_s\overline a_s+\overline a_{s-1}a_{s-1}=0$ we can successively move the
$\ol a_r$'s to the left and obtain
$$
\ol a_{i-1}\ol a_{i-2}\dots \ol a_{0}\epsilon a_0 a_1\dots
a_{\ell}\ol a_{\ell}\ol a_{\ell-1}\dots \ol a_{j}
= \pm\,
\ol a_{i-1}\ol a_{i-2}\dots\ol a_0\epsilon(a_0\ol a_0)^{\ell-j+1}
a_0a_1\dots a_{j-1}.
$$
Moreover, using that $a_0\ol a_0$ commutes with $\epsilon$ and then
moving the $a_r$'s to the left we get
\begin{eqnarray*}
\ol a_{i-1}\ol a_{i-2}\lefteqn{\dots \ol a_{0}\epsilon a_0 a_1\dots
a_{\ell-1}a_{\ell}\ol a_{\ell}\ol a_{\ell-1}\dots \ol a_{j}}\\
&=& \pm\,
\ol a_{i-1}\ol a_{i-2}\dots\ol a_0(a_0\ol a_0)^{\ell-j+1}\epsilon a_0a_1\dots a_{j-1}\\
&=& \pm\,
\ol a_{i-1}\ol a_{i-2}\dots(\ol a_0a_0)^{\ell-j+1}\ol a_0\epsilon a_0a_1\dots a_{j-1}\\
&=&\pm\,
\begin{cases}
a_ia_{i+1}\dots a_{i+\ell-j}\ol a_{i+\ell-j} \ol a_{i+\ell-j-1}\dots \ol a_1\ol a_0\epsilon a_0a_1\dots a_{j-1}&\mbox{ if } i+\ell-j\leq n-2\\ 0&\mbox{ else }\end{cases}
\end{eqnarray*}

The following result provides explicit vector space bases for the
deformed preprojective algebras $L_n^p$ of type $L$.
Note that the bases do not involve the deformation
polynomial $p$, i.e. the bases is independent of the polynomial.

\begin{Proposition}\label{abasis}
A $K$-basis of $L_n^p$ is given by the following paths
between the vertices $i$ and $j$, where $i,j\in\{0,1,\ldots,n-1\}$.
$$\begin{array}{ll}
(1)\mbox{~~}a_ia_{i+1}\dots a_{j-1}
& \mbox{for~~} i<j \\
(2)\mbox{~~}a_ia_{i+1}\dots a_{j-1}a_j\dots
a_\ell\ol a_\ell \ol a_{\ell-1}\dots \ol a_j
& \mbox{for~~} i<j\mbox{~~and some~~}j\le \ell\le n-2 \\
(3)\mbox{~~}\ol a_{i-1}\ol a_{i-2}\dots \ol a_{j}
& \mbox{for~~}i\ge j \\
(4)\mbox{~~}a_ia_{i+1}\dots  a_\ell\ol a_\ell \ol a_{\ell-1}\dots
\ol a_i\dots\ol a_j
& \mbox{for~~}i \ge j\mbox{~~and some~~}i\le \ell\le n-2 \\
(5)\mbox{~~}
\ol a_{i-1}\ol a_{i-2}\dots \ol a_{0}\epsilon a_0 a_1\dots a_{j-1}
& \mbox{for any~~}i,j \\
(6)\mbox{~~}
\ol a_{i-1}\ol a_{i-2}\dots \ol a_{0}\epsilon a_0 a_1\dots
a_{\ell-1}a_{\ell}\ol a_{\ell}\ol a_{\ell-1}\dots \ol a_{j}
& \mbox{for~~} i<j\mbox{~~and some~~}j\le \ell \le n-2 \\
(7)\mbox{~~}
a_ia_{i+1}\dots a_\ell\ol a_\ell \ol a_{\ell-1}\dots \ol a_1\ol a_0\epsilon a_0a_1\dots a_{j-1}
& \mbox{for~~}i\ge j\mbox{~~and some~~}i\le \ell\le n-2
\end{array}
$$
\end{Proposition}

\begin{Remark} \label{Rem-abasis}
{\em
(1) In type (3) the case $i=j$ yields an empty product
which has to be interpreted as the
trivial paths $e_i$ for every vertex $i$.

(2) The longest paths in this basis of $L_n^p$ are of length
$2n-1$, occurring in (5) for $i=j=n-1$ and in
(7) for $i=j\in\{0,1,\ldots,n-2\}$ and $l=n-2$,
respectively.
These elements span the socles of the projective
indecomposable modules corresponding to the vertices $i\in\{0,1,\ldots,n-1\}$.

(3) The socle element of the projective indecomposable module
corresponding to the vertex $0$ can also be expressed in terms of
powers of the loop $\epsilon$ (recall that $\epsilon^{2n}=0$
in the algebra $L_n^p$). In fact it is not hard to check that we have
$\epsilon^{2n-1} =
a_0a_{1}\dots a_{n-2}\ol a_{n-2} \ol a_{n-3}\dots \ol a_1\ol a_0\epsilon.
$
(Note that lower powers of $\epsilon$ are not necessarily occuring as paths
in the above list, but are linear combinations of these, the precise shape
depending on the deformation polynomial $p$.)
In particular, the socle elements are precisely the basis elements
having length $2n-1$.

(4) The above basis seems very suitable for making the following inductive
proof work. However, later in the paper we will also use slightly
different bases involving powers of the loop $\epsilon$.
}
\end{Remark}

\begin{proof}
The above discussion proves that the given elements form a generating set.
We need to show that these elements are linearly independent.
Since the defining relations of the algebra $L_n^p$
are relations between closed paths, we may suppose that
a linear combination of paths starting at $i$ and ending at $j$ is $0$.
By symmetry we may suppose that $i\geq j$ and hence we get a linear combination
\begin{eqnarray*} \label{lincomb}
0&=&\nu_0\cdot\ol a_{i-1}\ol a_{i-2}\dots \ol a_{j}+
\nu_1\cdot\ol a_{i-1}\ol a_{i-2}\dots \ol a_{0}\epsilon a_0 a_1\dots a_{j-1}+\\
&&+\sum_{\ell=i}^{n-2} \lambda_\ell\cdot a_ia_{i+1}\dots
a_\ell\ol a_\ell \ol a_{\ell-1}\dots \ol a_{i+1}\ol a_i\dots\ol a_j+\\
& & +\sum_{\ell=i}^{n-2}\mu_\ell\cdot a_ia_{i+1}\dots
a_\ell\ol a_\ell \ol a_{\ell-1}\dots \ol a_1\ol a_0\epsilon a_0a_1\dots a_{j-1}\\
\end{eqnarray*}
with scalars $\nu_0,\nu_1,\lambda_{\ell},\mu_{\ell}\in K$.
Note that the paths occurring have the following lengths:
length $i-j$ for the summand of type (3) with coefficient $\nu_0$,
length $i+j+1$ for the summand of type (5) with coefficient $\nu_1$,
length $2\ell -i-j+2$ for the summand of type (4) with coefficient $\lambda_{\ell}$,
and length $2\ell -i+j+3$ for the summand of type (7) with coefficient $\mu_{\ell}$.

Denote by $J=J(L_n^p)$ the two-sided ideal of $L_n^p$ generated by the arrows
of the quiver.

From the lengths of the paths we observe that all summands in the above expression
are contained in
$J^{i-j+1}$, except the one with coefficient $\nu_0$. So considering
the above equation modulo $J^{i-j+1}$ we can deduce that $\nu_0=0$.

The remaining summands are given by paths which pass through the vertex $i+1$
(recall that $\ell\ge i$),
except for the summand with coefficient $\nu_1$. So considering the above
expression modulo the two-sided ideal $L_n^pe_{i+1}L_n^p$ we get that
also $\nu_1=0$.

Hence we are left to consider the equation
$$0=\sum_{\ell=i}^{n-2} a_ia_{i+1}\dots  a_\ell\ol a_\ell \ol a_{\ell-1}\dots
\ol a_{i+1}\ol a_i\dots\ol a_j \left(\lambda_\ell+\mu_\ell
\ol a_{j-1}\ol a_{j-2}\dots\ol a_1\ol a_0\epsilon a_0a_1\dots a_{j-1}
\right).$$
We shall prove by induction on $\ell$ that all coefficients are $0$.
Observe that for each $\ell=i,\ldots,n-2$ the paths with coefficients
$\lambda_{\ell}$ and $\mu_{\ell}$ pass through the vertex $\ell+1$
but not through the vertex $\ell+2$.

For $\ell =i$ we consider the above equation modulo the two-sided ideal
$L_n^pe_{i+2}L_n^p$ and obtain that
$$0= \lambda_i\cdot a_i\ol a_{i}\ol a_{i-1}\dots\ol a_j
+ \mu_i\cdot
a_i\ol a_{i}\ol a_{i-1}\dots\ol a_j\dots\ol a_0\epsilon a_0a_1\dots a_{j-1}
$$
Since the first path is strictly shorter than the second we again
consider the equation modulo a suitable power of the ideal $J$
and can deduce that $\lambda_i=0$, and then also that $\mu_i=0$.

By a completely analogous argument we can immediately deduce inductively
that all coefficients $\lambda_{\ell}$ and $\mu_{\ell}$ are 0.
\end{proof}

\begin{Remark} \label{Rem-Cartan-dim}
{\em
From Proposition~\ref{abasis} one can derive
the Cartan matrix of the deformed preprojective algebras
of type $L_n$. Since the basis is independent of the
deformation polynomial $p$
the Cartan matrix of $L_n^0$ and $L_n^p$ coincide;
this has already been observed by
Bia\l kowski, Erdmann and Skowro\'nski
\cite[Lemma 3.2]{BES}.
The Cartan matrix $C_n$ of the deformed preprojective
algebras of type $L_n$ actually has the following form
$$C_n=2\cdot \left(
\begin{array}{ccccc}
n & n-1 & \ldots & 2 & 1 \\
n-1 & n-1 & \ldots & 2 & 1 \\
\vdots & \vdots & \ddots & \vdots & \vdots \\
2 & 2 & \ldots & 2 & 1 \\
1 & 1 & \ldots & 1 & 1
\end{array}
\right)
$$
From this shape one easily computes that the determinant
of the Cartan matrix is $\det C_n = 2^n$ for all $n\in\mathbb{N}$.
Moreover, the vector space dimension of $L_n^p$ is $\frac{1}{3}n(n+1)(2n+1)$
for all $n\in\mathbb{N}$.
}
\end{Remark}

\subsection{Deformed preprojective algebras of type $L$ are symmetric}
\label{Sec:symmetric}
The aim of this section is to show that for any deformation
polynomial $p\in K[X]$ and any $n\in\mathbb{N}$ the deformed
preprojective algebra $L_n^p$ is a symmetric algebra. This is a result of
Bia\l kowski, Erdmann and Skowro\'{n}ski \cite{BESIII}, as announced in
\cite{Bialkowski-ICRA-Abstract}.
Since this result is not yet available in the literature we
include an independent proof in this section for the convenience of the reader.

According to Lemma \ref{propertiesofNakayamatwistedcentre}
it suffices to show that the
identity is a Nakayama automorphism for the algebra $L_n^p$.
Recall from Proposition \ref{thenuformula} that a Nakayama automorphism $\nu$ for a
self-injective algebra $A$ over a field $K$ is characterized by the property
$\langle a,b\rangle = \langle b,\nu(a)\rangle$ for all $a,b,\in A$ where
$\langle.\,,.\rangle$ is a non-degenerate associative $K$-bilinear form on $A$.

The following general observation turns out to be useful when verifying
that a certain automorphism is indeed a Nakayama automorphism; namely,
it suffices to check the crucial property
on algebra generators of $A$.

\begin{Lemma} \label{Lemma-Nakayama-alggen}
Let $A$ be a self-injective algebra, with a non-degenerate associative
$K$-bilinear form $\langle.\,,.\rangle$. If an automorphism $\nu$ of $A$ satisfies
$\langle a_i,b\rangle = \langle b,\nu(a_i)\rangle$ for a
set of algebra generators $\{a_1,\ldots,a_r\}$ and all $b\in A$
then $\nu$ is a Nakayama automorphism of $A$.
\end{Lemma}

\begin{proof}
Every element of $A$ can be expressed as a product of the algebra generators.
We show that $\langle a,b\rangle = \langle b,\nu(a)\rangle$ for all
$a,b\in A$ by induction on the length of such an expression for $a$.
For any algebra generator $a_j$ and $a,b\in A$ we have
$$
\langle a a_j,b\rangle =  \langle a,a_jb\rangle = \langle a_jb,\nu(a)\rangle
= \langle a_j,b\nu(a)\rangle
 =  \langle b\nu(a), \nu(a_j)\rangle = \langle b,\nu(a)\nu(a_j)\rangle =
\langle b, \nu(aa_j)\rangle
$$
where for the first, third and fifth equality we used the associativity
of the form, for the second we used the induction hypotheses,
for the fourth equality we used the assumption on $\nu$
for algebra generators, and the last equality holds because $\nu$ is
an algebra homomorphism.
\end{proof}

Recall from Proposition \ref{prop:form} the construction of an
associative non-degenerate bilinear form on a self-injective
algebra, depending on the choice of a suitable basis. For a basis
$\mathcal{B}$ consisting of non-zero distinct paths and containing
a basis of the socle this bilinear form has been defined on basis
elements by $\langle a,b\rangle = \psi(ab)$ where
$\psi(x)= \left\{
\begin{array}{ll} 1 & \mbox{if~}x\in\soc(A)\cap\mathcal{B} \\
0 & \mbox{if~}x\not\in\soc(A)\cap\mathcal{B}
\end{array}
\right. .
$

For our aim of proving that the identity is a Nakayama automorphism
for $L_n^p$ we shall show that the bilinear form $\langle .,.\rangle$
corresponding to the basis $\mathcal{B}$
given in Proposition \ref{abasis} is indeed symmetric.
By the previous lemma we therefore have to verify that
$\langle a,b\rangle = \langle b,a\rangle$
for every algebra generator
$a\in \{e_0,\ldots,e_{n-1},\epsilon, a_0,\ldots,a_{n-2},
\ol a_0,\ldots,\ol a_{n-2}\}$,
and $b$ running through the basis of Proposition \ref{abasis}.

It is immediate from the definition of the form $\langle.,.\rangle$
that $\langle e_i,b\rangle = \langle b,e_i\rangle$; in fact, the value
on either side is 1 precisely if $b$ is a basis element from the socle,
and 0 otherwise.

So it remains to deal with the cases where $a$ is an arrow of the quiver
of $L_n^p$.

We start with the loop $\epsilon$.
By definition the value in both $\langle \epsilon,b\rangle$ and
$\langle b,\epsilon\rangle$ is 0 unless $b\in e_0 L_n^p e_0$.
By Proposition \ref{abasis}, for the latter space a basis is
given by the elements $e_0$, $\epsilon$,
$a_0\ldots a_{\ell} \ol a_{\ell}\ldots \ol a_0$ and
$a_0\ldots a_{\ell} \ol a_{\ell}\ldots \ol a_0\epsilon$ where
$0\le \ell\le n-2$. Using the defining relations
$a_s\ol a_s + \ol a_{s-1} a_{s-1} =0$
it is not difficult to see that in $L_n^p$ we have
$a_0\ldots a_{\ell} \ol a_{\ell}\ldots \ol a_0 = \pm (a_0\ol a_0)^{\ell +1}$.
From this we can deduce, by using the relation
$\epsilon^2+\epsilon^3p(\epsilon)+ a_0\ol a_0=0$,
that $\epsilon$ commutes with every element of $e_0L_n^pe_0$.
But then we clearly have for all $b\in e_0L_n^p e_0$ that
$$\langle \epsilon,b\rangle = \psi(\epsilon b) = \psi(b\epsilon)
= \langle b,\epsilon\rangle.
$$
We now consider the case $a=a_r$ (for some $0\le r\le n-2$).
Again by definition the value in both $\langle a_r,b\rangle$ and
$\langle b,a_r\rangle$ is 0 unless $b\in e_{r+1} L_n^p e_r$.
Moreover, for a basis element $b\in e_{r+1} L_n^p e_r$
the value in both $\langle a_r,b\rangle$ and
$\langle b,a_r\rangle$ is also 0 unless $a_rb$ (resp. $ba_r$)
is a nonzero element in the socle of $L_n^p$.
According to Remark \ref{Rem-abasis}\,(2) we know that $a_rb$ and $ba_r$
can only be a nonzero element in the socle if $b$ is a path of length $2n-2$.
However, it is immediately checked that the only basis element
$b\in e_{r+1}L_n^pe_r$ in Proposition \ref{abasis} of length
$2n-2$ is
$$b=a_{r+1}a_{r+2}\ldots a_{n-2}\ol a_{n-2} \ol a_{n-3}\ldots \ol a_0
\epsilon a_0a_1\ldots a_{r-1}.
$$
For this element we have that $a_rb$ is the socle element in $\mathcal{B}$
(of type (7))
corresponding to vertex $r$ and that $ba_r$ is the socle element
corresponding to vertex $r+1$ (of type (7) if $r<n-2$ and of
type (5) if $r=n-2$), i.e. we can deduce
$$\langle a_r,b\rangle = \psi(a_rb) = 1 = \psi(ba_r) =
\langle b,a_r\rangle.
$$
Since by the above remarks in all other cases for $b$
both values $\langle a_r,b\rangle$ and
$\langle b,a_r\rangle$ vanish we get the desired
statement $\langle a_r,b\rangle = \langle b,a_r\rangle$
for all basis elements $b$ from Proposition \ref{abasis}.
\smallskip

Finally we consider the case where $a=\ol a_r$ for some $0\le r\le n-2$.
This is mainly analogous to the previous case but at a certain
point pointed out below one has to be careful.
Again by definition the value in both $\langle \ol a_r,b\rangle$ and
$\langle b,\ol a_r\rangle$ is 0 unless $b\in e_{r} L_n^p e_{r+1}$
and $\ol a_rb$ and $b\ol a_r$
are nonzero elements in the socle.
By Remark \ref{Rem-abasis}\,(2) the products $\ol a_rb$ and $b\ol a_r$
can only be nonzero elements in the socle if $b$ is a path of length $2n-2$.
The only basis element
$b\in e_{r}L_n^pe_{r+1}$ of length $2n-2$ in Proposition \ref{abasis}
occurs in type (6) (for $r<n-2$) and in type (5)
(for $r=n-2$) and has the form
$$b=\ol a_{r-1}\ol a_{r-2}\ldots \ol a_{0}\epsilon a_{0} \ldots a_{n-2}
\ol a_{n-2}\ldots \ol a_{r+1}.
$$
Now, when calculating the values of $\langle \ol a_r,b\rangle$ and
$\langle b,\ol a_r\rangle$ one has to be careful since the products
$$\ol a_rb = \ol a_r\ol a_{r-1}\ol a_{r-2}\ldots \ol a_{0}\epsilon a_{0}
\ldots a_{n-2}\ol a_{n-2}\ldots \ol a_{r+1}
\in e_{r+1}L_n^p e_{r+1}
$$
and
$$b\ol a_r = \ol a_{r-1}\ol a_{r-2}\ldots \ol a_{0}\epsilon a_{0}
\ldots a_{n-2}\ol a_{n-2}\ldots \ol a_{r+1}\ol a_r
\in e_{r}L_n^p e_{r}
$$
are not elements of the basis $\mathcal{B}$ given in Proposition
\ref{abasis}; the only exception is for $r=n-2$ where $\ol a_{n-2}b$
is a basis element of type (5). However, Lemma \ref{nichtumkehren}\,(1) below
shows how to express these in terms of the basis $\mathcal{B}$; namely
for all $r\in \{0,\ldots,n-2\}$ we have that $\ol a_r b$ and
$b\ol a_r$ are equal (not only up to a scalar\,!) to the socle elements
occurring in the basis $\mathcal{B}$. Hence we obtain that
$\langle \ol a_r,b\rangle = \psi(\ol a_rb) = 1
= \psi(b\ol a_r) = \langle b,\ol a_r\rangle$,
as desired.
\smallskip

Summarizing our above arguments we have now shown that
the associative non-degenerate bilinear form $\langle.,.\rangle$
corresponding (in the sense of \ref{prop:form}) to the
basis $\mathcal{B}$ of Proposition \ref{abasis}
is symmetric. We therefore have given an independent proof
of the following result, which is due to
Bia\l kowski, Erdmann and Skowro\'{n}ski \cite{BESIII}.

\begin{Theorem} \label{thm-symmetric}
Let $K$ be a field (of any characteristic).
For all $n\in \mathbb{N}$ and every polynomial $p\in K[X]$ the
deformed preprojective algebra $L_n^p$ is a symmetric algebra.
\end{Theorem}

We complete this section by providing an auxiliary result on relations in the
algebras $L_n^p$; the first part provides the missing calculations
in the last part of the above proof, the second part will be used in later sections.

\begin{Lemma} \label{nichtumkehren}
The following identities hold in the algebra $L_n^p$.
\begin{enumerate}
\item[{(a)}] For all $r\in\{0,\ldots,n-2\}$ we have that
$$\ol a_{r-1}\ol a_{r-2}\ldots \ol a_{0}\epsilon a_{0}
\ldots a_{n-2}\ol a_{n-2}\ldots \ol a_{r+1}\ol a_r
=
a_ra_{r+1}\dots a_{n-2}\ol a_{n-2} \ol a_{n-3}\dots
\ol a_0\epsilon a_0a_1\dots a_{r-1}
$$
\item[{(b)}] For all $l\in\{0,1,\dots,n-2\}$ we have that
$\ol a_{n-2}\ol a_{n-3}\dots\ol a_l a_l=0$.
\end{enumerate}
\end{Lemma}

\begin{proof}
(a) We shall use frequently the relations $a_s\ol a_s +\ol a_{s-1}a_{s-1} =0$
for $s\in\{1,\ldots,n-2\}$ and carefully keep track of the signs
occurring.

In the expression on the left hand side of assertion (a) we start by successively
moving $\ol a_{n-2},\ldots,\ol a_r$ to the left (but still right of $\epsilon$);
note that each such move the
gives a minus sign. Setting $c:=(n-2)+(n-3)+\ldots+(r+1)+r$ for abbreviation
we obtain that
\begin{equation} \label{eqn1}
\ol a_{r-1}\ol a_{r-2}\ldots \ol a_{0}\epsilon a_{0}
\ldots a_{n-2}\ol a_{n-2}\ldots \ol a_{r+1}\ol a_r
=
(-1)^{c}\,
\ol a_{r-1}\ol a_{r-2}\ldots \ol a_{0}\epsilon
(a_0\ol a_0)^{n-r-1} a_0a_1\ldots a_{r-1}.
\end{equation}
It follows directly from the defining relation
$\epsilon^2+\epsilon^3p(\epsilon)+a_0\ol a_0=0$
that $\epsilon$ commutes with $a_0\ol a_0$, so the
expression on the right hand side of (\ref{eqn1}) is equal to
\begin{equation} \label{eqn2}
(-1)^{c}\,
\ol a_{r-1}\ol a_{r-2}\ldots \ol a_{0}
(a_0\ol a_0)^{n-r-1} \epsilon a_0a_1\ldots a_{r-1}.
\end{equation}
The part to the right of $\epsilon$ already has the desired shape.
To the left of $\epsilon$ we now successively move the $a_0$'s to the
left; for the first $a_0$ we need $r$ such moves and obtain
that the expression in (\ref{eqn2}) equals
\begin{equation} \label{eqn3}
(-1)^{c}(-1)^r a_r \ol a_{r}\ol a_{r-1}\ldots \ol a_1 \ol a_0
(a_0\ol a_0)^{n-r-2} \epsilon a_0a_1\ldots a_{r-1}.
\end{equation}
For moving the next $a_0$ we need $r+1$ moves etc and eventually
get another sign of $(-1)^c$; more precisely the expression in
(\ref{eqn3}) is equal to
\begin{equation}
(-1)^{c}(-1)^{c} a_ra_{r+1}\dots a_{n-2}\ol a_{n-2} \ol a_{n-3}\dots
\ol a_0\epsilon a_0a_1\dots a_{r-1}
\end{equation}
where the signs cancel so that this
is precisely the right hand side in the assertion of part (a)
of the lemma.
\smallskip

(b) We show this by reverse induction on $l$. For $l=n-2$ this is just the
relation $\ol a_{n-2}a_{n-2}=0$. For $l<n-2$ we use the defining relation
$\ol a_l a_l=a_{l+1} \ol a_{l+1}$ and obtain
$$\ol a_{n-2}\ol a_{n-3}\dots\ol a_{l+1}\ol a_l a_l=
\ol a_{n-2}\ol a_{n-3}\dots\ol a_{l+1}a_{l+1} \ol a_{l+1}$$
where the latter is zero by induction hypothesis.
\end{proof}

\subsection{Linking $L_{n+1}^p$ and $L_n^p$}

For proving statements about the algebras $L_n^p$
we shall often argue by induction and then the following result
will turn out to be useful. As usual we denote the trivial path
of length zero corresponding to the vertex $i$
by $e_i$.

\begin{Lemma}\label{induction} For any $n\ge 1$,
there is an algebra epimorphism
$\pi_n:L_{n+1}^p\lra L_{n}^p$
satisfying
$$\pi_n(e_i)=e_i\mbox{~for all~ $0\le i\le n-1$}~~,~~
\pi_n(e_n)=0, ~\pi_n(\epsilon)=\epsilon,
$$
$$\;\;\pi_n(a_i)=a_i,\;\pi_n(\ol a_i)=\ol a_i\;
\mbox{~for all~ $0\le i\le n-2$~},\;\;\pi_n(a_{n-1})=0,\;\;\pi_n(\ol a_{n-1})=0.$$
Moreover, $\pi_n$ induces an algebra isomorphism
$$L_{n+1}^p/(L_{n+1}^pe_nL_{n+1}^p)\simeq L_n^p\;.$$
\end{Lemma}

\begin{proof}
The map $\pi_n$ is well-defined
since the defining relations
for the algebra $L_{n+1}^p$ are clearly verified in $L_n^p$
(perhaps the only not entirely obvious check is that
$\pi_n(a_{n-1}\overline{a}_{n-1} + \overline{a}_{n-2}a_{n-2})
= \pi_n(a_{n-1})\pi_n(\overline{a}_{n-1}) + \pi_n(\overline{a}_{n-2})\pi_n(a_{n-2})
= 0 + \overline{a}_{n-2}a_{n-2}$
which is zero in $L_n^p$).

For the second statement
we need to determine the kernel of $\pi_n$ (since $\pi_n$ is surjective by definition).
By definition of $\pi_n$ we have that $L_{n+1}^pe_nL_{n+1}^p$ is contained
in the kernel. On the other hand, the dimension of the kernel is the
difference of the dimensions of the algebras
$L_{n+1}^p$ and $L_n^p$. These are given in Remark \ref{Rem-Cartan-dim} and we
obtain
$$\dim\ker\pi_n = \frac{1}{3}(n+1)(n+2)(2n+3) - \frac{1}{3}n(n+1)(2n+1)
= 2(n+1)^2.
$$
However, in the basis of $L_{n+1}^p$ provided in Proposition~\ref{abasis},
there are already $2(n+1)^2$ basis elements which pass through the vertex $n$,
i.e. are contained in $L_{n+1}^p e_n L_{n+1}^p$. (More precisely, there are
$n$ such paths of type (1), $\frac{n(n-1)}{2}$ of type (2),
$n+1$ of type (3), $\frac{n(n+1)}{2}$ of type (4), $2n+1$ of type (5),
$\frac{n(n-1)}{2}$ of type (6) and $\frac{n(n+1)}{2}$ of type (7), respectively.)

Since $L_{n+1}^pe_nL_{n+1}^p\subseteq \ker\pi_n$ and dimensions agree,
the second claim of the lemma follows.
\end{proof}

\begin{Remark}
{\em
Since $\epsilon^{2n}=0$ in $L_n^p$, also $\epsilon^{2n}$ is in the
kernel of $\pi_n$. The dimension arguments in the above proof thus
show that $\epsilon^{2n}$ is contained in $L_{n+1}^pe_nL_{n+1}^p$,
i.e., $\epsilon^{2n}$ is a linear combination of paths passing through
the vertex $n$ (which could also be checked directly).
}
\end{Remark}

\subsection{Generating the commutator subspace}
\label{Sec:commutator}
We start with some computations on the basis elements occurring in
Proposition~\ref{abasis}.

The basis in Lemma~\ref{abasis} of $L_n^p$ is actually a union
of bases of $e_iL_n^pe_j$ for $i,j\in\{0,1,\dots,n-1\}$. We
consider the case $i=j$ in Lemma~\ref{abasis}. Only the basis elements of type
(3), (4), (5) and (7) admit $i=j$.
Up to signs (which are not essential since we are only interested
in generating sets) we have
\begin{eqnarray*}
(a_i\ol a_i)^2&=&\pm\,a_ia_{i+1}\ol a_{i+1}\ol a_i\\
(a_i\ol a_i)^3&=& \pm\,a_ia_{i+1}a_{i+2}\ol a_{i+2}\ol a_{i+1}\ol a_i\\
\dots&&\dots\\
(a_i\ol a_i)^{\ell+1}&=& \pm\,
a_ia_{i+1}\dots a_{i+\ell}\ol a_{i+\ell}\dots\ol a_{i+1}\ol a_i\\
(a_i\ol a_i)^{\ell+2}&=& \pm\, (a_i\ol a_i)
(a_ia_{i+1}\dots a_{i+\ell}\ol a_{i+\ell}\dots\ol a_{i+1}\ol a_i)\\
&=& \pm\,
a_ia_{i+1}\dots a_{i+\ell}a_{i+\ell+1}\ol a_{i+\ell+1}\ol a_{i+\ell}\dots
\ol a_{i+1}\ol a_i
\end{eqnarray*}
for all $\ell$.
Hence the basis elements of type (4) for $i=j$ can be expressed as
$\pm(a_i\ol a_i)^m$ for certain $m$. In particular, for $i=0$ one gets
$$(a_0\ol a_0)^m=\pm\,\left(\epsilon^{2m}(1+\epsilon p(\epsilon))^m\right).$$

Moreover, we see that
\begin{eqnarray*}
[a_i,a_{i+1}\dots
\lefteqn{a_\ell\ol a_\ell \ol a_{\ell-1}\dots \ol a_1\ol a_0\epsilon a_0a_1\dots a_{i-1}]=}\\
&=&a_ia_{i+1}\dots a_\ell\ol a_\ell \ol a_{\ell-1}\dots \ol a_1\ol a_0\epsilon a_0a_1\dots a_{i-1}
-a_{i+1}\dots a_\ell\ol a_\ell \ol a_{\ell-1}\dots \ol a_1\ol a_0\epsilon a_0a_1\dots a_{i-1}a_i
\end{eqnarray*}
$\mbox{for some~~}i\le \ell\le n-2$.
Hence, two different basis elements of type (7) for $i=j$ differ by a commutator.
Therefore, modulo commutator $[L_n^p,L_n^p]$ we need to consider the
basis elements of type (7) only for $i=j=0$:
$$
a_0a_{1}\dots a_{\ell-1}a_\ell\ol a_\ell \ol a_{\ell-1}\dots \ol a_1\ol a_0\epsilon
 \mbox{~~for some~~}0\le \ell\le n-2
$$
But now,
$$
a_0a_{1}\dots a_{\ell-1}a_\ell\ol a_\ell \ol a_{\ell-1}\dots \ol a_1\ol a_0\epsilon =\pm
(a_0\ol a_0)^{\ell+1}\epsilon=\pm\epsilon(\epsilon^2+\epsilon^3p(\epsilon))^{\ell+1}=
\epsilon^{2\ell+3}(1+\epsilon p(\epsilon))^{\ell+1}
$$
Therefore, the basis of $e_0L_n^pe_0$ in Proposition~\ref{abasis}
consists of one element of type (3), one element of type (5), elements of type (7) which have the form $\epsilon^{2\ell+3}(1+\epsilon p(\epsilon))^{\ell+1}$
and elements of type (4) of the form $\epsilon^{2\ell+2}(1+\epsilon p(\epsilon))^{\ell+1}$. Hence the set
$$\{e_0\}\cup\{\epsilon\}\cup \{\epsilon^{2\ell+3}(1+\epsilon p(\epsilon))^{\ell+1}\;|\;0\leq\ell<n-1\}\cup
\{\epsilon^{2\ell+2}(1+\epsilon p(\epsilon))^{\ell+1}\;|\;0\leq\ell<n-1\}$$
forms a basis of $e_0L_n^pe_0$.

\begin{Lemma}\label{epsilonpower}
The set $\{\epsilon^\ell\;|\;0\le\ell<2n\}$ is a $K$-basis of $e_0L_n^pe_0$.
\end{Lemma}

\begin{Remark}
{\em
Of course, we put $\epsilon^0=e_0$ in Lemma~\ref{epsilonpower}.
}
\end{Remark}

\begin{proof}
We know that the set
$${\mathcal S}:=\{e_0\}\cup\{\epsilon\}\cup\{\epsilon^{2\ell+3}(1+\epsilon p(\epsilon))^{\ell+1}\;|\;0\leq\ell<n-1\}\cup
\{\epsilon^{2\ell+2}(1+\epsilon p(\epsilon))^{\ell+1}\;|\;0\leq\ell<n-1\}$$
forms a basis of $e_0L_n^pe_0$.
Expressing these elements as linear combinations of the set $\{\epsilon^\ell\;|\;0\le\ell<2n\}$
one obtains a square upper triangular matrix with diagonal entries $1$. Hence since
$\mathcal S$ is a basis, also
$\{\epsilon^\ell\;|\;0\le\ell<2n\}$ is a $K$-basis of $e_0L_n^pe_0$.
\end{proof}

\begin{Lemma}\label{basisofcommutatorquotient}
For all $n\ge 0$ the factor space
$L_{n+1}^p/[L_{n+1}^p,L_{n+1}^p]$ has a $K$-linear generating set
$$\{e_0,e_1,\dots,e_n\}\cup \{\epsilon^{2m+1}\,|\,0\le m\le n\}.$$
\end{Lemma}

\begin{proof}
It is a general fact that every non-closed path (i.e. a path with
different start and end point) is a commutator; in fact, take
the commutator with the trivial path corresponding to the start point
(or end point).
Moreover, it is easy to see that the cosets of the trivial paths are
always linearly independent modulo the commutator subspace.

In view of Lemma~\ref{induction} we only need to detect elements outside
the commutator of $L_n^p$ which become a commutator in $L_{n+1}^p$ and determine
which elements of $L_{n+1}^pe_nL_{n+1}^p$ are commutators.

We shall proceed by induction on $n$.
The lemma is clearly true for $n=0$. For $n>0$ we can use the
list of closed paths given in Proposition \ref{abasis}.
We shall start by identifying certain closed paths as being commutators.

$\bullet$ For all $m\in\{0,1,\ldots,n-1\}$ we have
\begin{eqnarray*}
a_ma_{m+1}\dots a_{n-1}\ol a_{n-1}\ol a_{n-2}\dots \ol a_m&=&
(a_ma_{m+1}\dots a_{n-1})(\ol a_{n-1}\ol a_{n-2}\dots \ol a_m)\\
& &-
(\ol a_{n-1}\ol a_{n-2}\dots \ol a_m)(a_ma_{m+1}\dots a_{n-1}) \\
& \in & [L_{n+1}^p,L_{n+1}^p]
\end{eqnarray*}
since $(\ol a_{n-1}\ol a_{n-2}\dots \ol a_m)(a_ma_{m+1}\dots a_{n-1})=0$ by
Lemma~\ref{nichtumkehren}.
\smallskip

$\bullet$ Using the defining relation $\ol a_{n-1}a_{n-1}=0$ in $L_{n+1}^p$
we have that
$$a_{n-1}\ol a_{n-1} = a_{n-1}\ol a_{n-1}-\ol a_{n-1}a_{n-1}
= [a_{n-1},\ol a_{n-1}] \in [L_{n+1}^p,L_{n+1}^p].
$$
Moreover, using the relation $\ol a_ia_i=a_{i+1}\ol a_{i+1}$
for all $i\in\{0,1,\ldots,n-2\}$ we have that
$$[a_i,\ol a_i]=a_i\ol a_i-\ol a_ia_i=a_i\ol a_i-a_{i+1}\ol a_{i+1}.$$
Inductively we can assume that
$a_{i+1}\ol a_{i+1}\in [L_{n+1}^p,L_{n+1}^p]$ and hence
we deduce that $a_i\ol a_i\in [L_{n+1}^p,L_{n+1}^p]$ for all $i\in\{0,1,\ldots,n-1\}$.

Moreover, consider a power $(a_i\ol a_i)^m$ for some integer $m\ge 2$.
Then for all $i\in\{0,1,\ldots,n-2\}$ we have
$$[a_i,(\ol a_ia_i)^{m-1}\ol a_i]=(a_i\ol a_i)^m-(\ol a_ia_i)^m
=(a_i\ol a_i)^m-(a_{i+1}\ol a_{i+1})^m.
$$
Inductively, we obtain that
$(a_i\ol a_i)^m \equiv (a_{n-1}\ol a_{n-1})^m \mod\; [L_{n+1}^p,L_{n+1}^p];$
but $(a_{n-1}\ol a_{n-1})^m =0 $ for $m\ge 2$ (using the defining relation
$\ol a_{n-1}a_{n-1}=0$).

Together with the above arguments for the case $m=1$ we can thus deduce
\begin{equation}
(a_{i}\ol a_{i})^m \in [L_{n+1}^p,L_{n+1}^p] \mbox{ for all $i\in\{0,1,\ldots,n-1\}$
and all $m\ge 1$}.
\end{equation}

\smallskip

$\bullet$ In particular the preceding arguments imply that
$(\epsilon^2+\epsilon^3p(\epsilon))^m=-(a_0\ol a_0)^m\in [L_{n+1}^p,L_{n+1}^p]$
for all integers $m\ge 1$.
\medskip

In a second step after showing certain closed paths to be commutators
we now examine (nontrivial) closed paths in $L_{n+1}^pe_nL_{n+1}^p$
and in particular determine the dimension of the image in the factor space
$L_{n+1}^pe_nL_{n+1}^p/[L_{n+1}^p,L_{n+1}^p]$.
According to Proposition~\ref{abasis}
there are two types of such paths, the long paths
$a_ia_{i+1}\ldots a_{n-1}\ol a_{n-1}\ldots \ol a_{0}\epsilon a_0\ldots a_{i-1}$
for $i\in\{0,1,\ldots,n\}$ corresponding to socle
elements and the short paths
$a_ia_{i+1}\dots a_{n-1}\ol a_{n-1}\dots\ol a_i$
for $i\in\{0,1,\ldots,n-1\}$.

For the latter we already observed
at the beginning of the proof that they are all commutators, i.e.
that $a_ia_{i+1}\dots a_{n-1}\ol a_{n-1}\dots\ol a_i\in [L_{n+1}^p,L_{n+1}^p]$
for all $i\in\{0,1,\ldots,n-1\}$.

For the former paths, corresponding to socle elements,
consider for $i\in\{1,\ldots,n\}$ the commutator
\begin{eqnarray*}
[a_i\ldots a_{n-1}\ol a_{n-1}\ldots \ol a_0\epsilon,a_0\ldots a_{i-1}]
& = &
a_ia_{i+1}\dots a_{n-1}\ol a_{n-1}\dots \ol a_0\epsilon a_0\dots a_{i-1} \\
& & -
a_0\dots a_{i-1}a_ia_{i+1}\dots a_{n-1}\ol a_{n-1}\dots \ol a_0\epsilon \\
& = &  a_ia_{i+1}\dots a_{n-1}\ol a_{n-1}\dots \ol a_0\epsilon a_0\dots a_{i-1}
- \epsilon^{2n+1}
\end{eqnarray*}
where for the last equation see Remark \ref{Rem-abasis}\,(3). Therefore,
all the long paths corresponding to socle elements are equivalent to
$\epsilon^{2n+1}$ modulo the commutator space.

\smallskip

Therefore, the image of $L_{n+1}^pe_nL_{n+1}^p$ in $L_{n+1}^p/[L_{n+1}^p,L_{n+1}^p]$
is $2$-dimensional with a basis given by the cosets of
$\{e_n,\epsilon^{2n+1}\}$.

The assertion of the lemma now follows by induction, using that
$L_{n+1}^p/\left(L_{n+1}^pe_nL_{n+1}^p\right) = L_n^p$
by Lemma \ref{induction}.
\end{proof}

\begin{Remark}
{\em
We did not yet prove that this generating set is actually a $K$-basis.
This fact is going to be shown in Proposition~\ref{explicitdescrofK}.
}
\end{Remark}

\subsection{The centre}

The aim of this section is to have a look at the centres of the
deformed preprojective algebras $L_n^p$ of type $L$.
The centre will be important to us by the following observation.
It is not so difficult to write down quite a lot of
commutators, as we have seen in Section \ref{Sec:commutator}.
It is however difficult in general to show that these
commutators actually generate the commutator space.
By the discussion in Section \ref{Nakayamatwistedcentre},
for a symmetric algebra $A$ we get
$$\dim_K A=\dim_K [A,A]+\dim_K Z(A).$$

Since the canonical projection $\pi_n:L_{n+1}^p\lra L_n^p$
from Lemma~\ref{induction}
is a surjective algebra homomorphism, the restriction of
$\pi_n$ to the centre $Z(L_{n+1}^p)$ induces a ring homomorphism
$Z(L_{n+1}^p)\lra Z(L_n^p).$ In principle, we could use this to determine
the centre of $L_n^p$ inductively although this might become quite technical.

Fortunately, with the methods developed in this paper we shall not really
need to determine the entire centre; it will turn out that it suffices
to find one central element whose powers generate a large enough central
subspace.

\medskip

\begin{Lemma}\label{generatingsetofcentre}
The following holds for the deformed preprojective algebra $L_n^p$.
\begin{enumerate}
\item[{(1)}] The element
$\epsilon^2+\epsilon^3p(\epsilon)+\sum_{\ell=0}^{n-3}(-1)^{\ell+1}\ol
a_{\ell}a_{\ell}$
is contained in the centre of $L_n^p$.
\item[{(2)}] The following subset is a $K$-free subset of the centre
$$\left\{\left(\epsilon^2+\epsilon^3p(\epsilon)+
\sum_{\ell=0}^{n-3} (-1)^{\ell+1}\ol a_{\ell}a_{\ell} \right)^s\;|\;
0\le s\le n-1\right\}\subseteq Z(L_n^p).$$
\item[{(3)}] $\soc(L_n^p)\subseteq Z(L_n^p)$.
\end{enumerate}
\end{Lemma}

\begin{proof}
(1) For proving that $\lambda:=\epsilon^2+\epsilon^3p(\epsilon)+
\sum_{\ell=0}^{n-3}(-1)^{\ell+1}\ol a_{\ell} a_{\ell}$
is a central element it is sufficient
to show that it commutes with the algebra generators of $L_n^p$.
This is clear for the trivial paths $e_i$ since all paths occurring in
$\lambda$ are closed paths.
For the loop $\epsilon$ we get
$\epsilon\lambda = \epsilon^3+\epsilon^4p(\epsilon) = \lambda\epsilon$.
For the arrow $a_0$ we have, using the defining relation
$\epsilon^2+\epsilon^3p(\epsilon)+a_0\ol a_0=0$, that
$$\lambda a_0 = (\epsilon^2+\epsilon^3p(\epsilon))a_0 = (-a_0\ol a_0)a_0
= a_0\lambda.
$$
Similarly, for the arrow $\ol a_0$ have
$$\lambda \ol a_0 = -\ol a_0 a_0\ol a_0 = \ol a_0(\epsilon^2+\epsilon^3p(\epsilon))
= \ol a_0\lambda.
$$
For the arrows $a_i$ where $1\le i\le n-2$, we use the relations
$\ol a_{i-1} a_{i-1}+a_i\ol a_i=0$ and get
$$\lambda a_i = (-1)^i \ol a_{i-1}a_{i-1}a_i = (-1)^{i+1} a_i\ol a_i a_i
= a_i\lambda.
$$
Finally, we get in a similar fashion for the arrows $a_i$ where $1\le i\le n-2$
that
$$\lambda \ol a_i = (-1)^{i+1} \ol a_{i} a_{i}\ol a_i = (-1)^{i+2}
\ol a_{i}\ol a_{i-1}  a_{i-1} = (-1)^{i}
\ol a_{i}\ol a_{i-1}  a_{i-1}
= \ol a_i\lambda.
$$

(2) The fact that the powers of $\epsilon^2+\epsilon^3p(\epsilon)+
\sum_{\ell=0}^{n-3}(-1)^{\ell+1}\ol a_{\ell}a_{\ell}$
form a linearly independent set comes from the
fact that the powers of $\epsilon$ form a linearly independent set (cf Lemma~\ref{epsilonpower}).

(3) This holds since multiplication from the left or right
of any basis element of the socle
with a path of length at least 1 gives 0. Moreover, the algebras $L_n^p$
are weakly symmetric, i.e. the socle has a basis
consisting of closed paths, so any socle element also commutes with
the trivial paths.
\end{proof}

We are now in the position to give a basis of the commutator space
$[L_n^{{p}},L_n^{{p}}]$ of the deformed preprojective algebras
$L_n^{{p}}$ of type $L$. As a consequence we can strengthen the
statement in Lemma \ref{basisofcommutatorquotient}; namely the
cosets of $\{e_0,\ldots,e_{n-1}\}\cup \{\epsilon^{2\ell+1}\,|\,0\le \ell \le n-1\}$
are a basis (and not only a generating set)
of the factor space $L_n^{{p}}/[L_n^{{p}},L_n^{{p}}]$.

\begin{Proposition}
\label{explicitdescrofK} For any polynomial $p(X)\in K[X]$
and $n\ge 2$ the following holds for the deformed preprojective
algebras $L_n^{p}$.
\begin{enumerate}
\item[{(a)}] The commutator space has dimension
$$\dim_K [L_n^{{p}},L_n^{{p}}] = \dim_K L_n^{{p}} - 2n
= \frac{1}{3} n(n-1)(2n+5).$$
\item[{(b)}] The centre has dimension
$\dim_K Z(L_n^{X^{p}}) = 2n.$
\item[{(c)}] The cosets of
$\{e_0,\ldots,e_{n-1}\}\cup \{\epsilon^{2\ell+1}\,|\,0\le \ell \le n-1\}$
form a basis of the factor space $L_n^{{p}}/[L_n^{X^{p}},L_n^{{p}}]$.
\item[{(d)}] In terms of the basis $\mathcal{B}$ of $L_n^{{p}}$ given in
Proposition \ref{abasis},
a $K$-basis of $[L_n^{{p}},L_n^{{p}}]$ is given by
\begin{itemize}
\item[{(i)}]
all non-closed paths in $\mathcal{B}$
(i.e. with starting vertex different from the ending vertex),
\item[{(ii)}] all closed paths in $\mathcal{B}$
of even length at least $2$
with starting vertex different from vertex $0$,
\item[{(iii)}]
the difference of two closed paths in $\mathcal{B}$
of equal odd length with consecutive
starting vertices $i$ and $i+1$ where $0\le i\le n-2$,
\item[{(iv)}] the elements
$a_0a_1\ldots a_{\ell} \ol a_{\ell}\ol a_{\ell-1}\ldots \ol a_0$
where $0\le \ell\le n-2$.
\end{itemize}
\end{enumerate}
\end{Proposition}

\begin{proof}
(a) For every
symmetric algebra $A$ we have
\begin{equation} \label{eq-dim-comm1}
\dim_K [A,A] = \dim_K A - \dim_K Z(A)
\end{equation}
since $\Hom_K(A/[A,A],K)\simeq Z(A)$ as $Z(A)$-modules.

In our situation for $A = L_n^{{p}}$ we get from Lemma
\ref{basisofcommutatorquotient} the lower bound
\begin{equation} \label{eq-dim-comm2}
\dim_K [A,A] \ge \dim_K L_n^{{p}} - 2n.
\end{equation}
We shall now produce sufficiently many linear independent
elements in the centre to obtain this also as
an upper bound.
The centre
$Z(L_n^p)$ contains the $K$-free subset
$$\left\{\left(\epsilon^2+\epsilon^3p(\epsilon)+
\sum_{\ell=0}^{n-3} (-1)^{\ell+1}\ol a_{\ell}a_{\ell} \right)^s\;|\;
0\le s\le n-1\right\}\subseteq Z(L_n^p)$$
of cardinality $n$. The $K$-vector
space generated by these elements intersects with $\soc(L_n^p)$
only in $\{0\}$ since $\epsilon^{2s+2}$ is not in the socle for
$s\in \{0,1,\dots,n-2\}$.
However, $\soc(L_n^p)$ belongs to the centre.
Hence we get that $Z(L_n^p)$ is of dimension at least $2n$.

Altogether, we get a lower bound for the dimension of the
centre, namely
$$\dim_K(Z(L_n^p))\geq (n-1)+n + 1 =2n.$$
Plugging this into formula (\ref{eq-dim-comm1}) and combining with
(\ref{eq-dim-comm2})
proves the first equality in part (a) of the lemma.

The second equality then follows by a direct calculation from the
formula for the dimension of $L_n^{p}$ given in
Remark \ref{Rem-Cartan-dim}.
\smallskip

(b) The statement in part (b) now follows directly from part (a) by using
formula (\ref{eq-dim-comm1}).
\smallskip

(c) Follows by combining part (a) and Lemma \ref{basisofcommutatorquotient}.
\smallskip

(d) For each of the types (i)-(iv) we shall first verify that all these
elements are actually contained in the commutator space, and then count
their number. At the end it will turn out that the total number of elements
in (i)-(iv) is ($\dim_K L_n^{p}-2n$), i.e. equal to the dimension of the
commutator space, cf. part (a).
Since the elements are linearly independent (being part of a basis),
the claim of part (d) then follows.

(i) Non-closed paths are always commutators (take the commutator with the
trivial path corresponding to the starting vertex).

The number of non-closed paths in $\mathcal{B}$ can be read off from
the Cartan matrix of $L_n^{p}$ given in Remark \ref{Rem-Cartan-dim}. Namely
as the dimension of $L_n^{p}$ minus the trace of the Cartan matrix,
i.e. we get $\dim_K L_n^{p} - n(n+1)$ non-closed paths in $\mathcal{B}$.

(ii) Such paths of even length
only occur in type (4) of Proposition \ref{abasis} and are of the form
$a_ia_{i+1}\ldots a_{\ell}\ol a_{\ell}\ol a_{\ell-1}\ldots \ol a_{i}$
where $i\neq 0$ and $i\le \ell\le n-2$.
Up to a sign, these paths are equal to $(a_i\ol a_i)^l$ (using the
relations $\ol a_r a_r + a_{r+1}\ol a_{r+1}$) and these
have been shown to be in the commutator space in the proof of Lemma
\ref{basisofcommutatorquotient}.

Summing over the possibilities for the various $i\neq 0$
there are $\frac{(n-1)(n-2)}{2}$ such paths.

(iii) Such a difference of closed paths with starting vertices $i$ and $i+1$ occurs
as a difference of a path of type (7) for vertex $i$ with a path of type
(5) or (7) for vertex $i+1$. More precisely,
these differences are of the form
$$a_ia_{i+1}\ldots a_{\ell}\ol a_{\ell}\ol a_{\ell-1}\ldots \ol a_0
\epsilon a_0\ldots a_{i-1} -
a_{i+1}\ldots a_{\ell}\ol a_{\ell}\ol a_{\ell-1}\ldots \ol a_0
\epsilon a_0\ldots a_{i},
$$
where $i\le \ell\le n-2$. This element is a commutator, namely
$[a_i, a_{i+1}\ldots a_{\ell}\ol a_{\ell}\ol a_{\ell-1}\ldots \ol a_0
\epsilon a_0\ldots a_{i-1}]$.

Summing over the possibilities for the various $i$
there are $\frac{n(n-1)}{2}$ such differences.

(iv) Up to a sign, these paths are equal to $(a_i\ol a_i)^l$ (using the
relations $\ol a_r a_r + a_{r+1}\ol a_{r+1}$) and these
have been shown to be in the commutator space in the proof of Lemma
\ref{basisofcommutatorquotient}.
Obviously, there are $n-1$ such paths.
\smallskip

The total number of elements in (i)-(iv) is easily computed to be
$\dim_K L_n^{X^{2j}} - 2n$, which is the dimension of the commutator space
by part (a). Thus part (d) follows.
\end{proof}

\section{The K\"ulshammer spaces and the main result}
\label{mainresultsection}

We now restrict to the case where the
deformation polynomial $p$ has the form $X^{2j}$ for some integer $j\ge 0$.
It has been shown in
\cite[Proposition 6.1]{BES} that the deformed algebras
$L_n^p$ for the polynomials $p=X^{2j}$ where $j\in\{0,1,\ldots,n-1\}$
form a family of pairwise non-isomorphic deformed preprojective algebras
of type $L$. Note that for all $j\ge n-1$ the algebra
$L_n^{X^{2j}}$ is the (undeformed) preprojective algebra of
type $L$; in fact, the only relation involving the polynomial $p$
reads
$$\epsilon^2 + a_0\ol a_0 +\epsilon^3p(\epsilon) =
\epsilon^2 + a_0\ol a_0 +\epsilon^{2j+3} = \epsilon^2 + a_0\ol a_0$$
because $\epsilon^{2n}=0$ in $L_n^p$.

Moreover, it has been announced \cite{Bialkowski-ICRA-Abstract}
that the algebras $L_n^{X^{2j}}$ for $j\in\{0,1,\ldots,n-1\}$
actually form a complete
list of representatives of the isomorphism classes of deformed preprojective
algebras of type $L$; details should appear in the forthcoming paper
\cite{BESIII}.

For the above reasons, focussing on the case of deformation
polynomials $p=X^{2j}$ is not really a restriction.

We continue to consider the deformed preprojective algebras
$A_n^j=L_n^{X^{2j}}$ over a field of characteristic $2$.
The K\"ulshammer spaces are defined as
$T_r(A_n^j)=\{x\in A_n^j\,|\,x^{2^r}\in [A_n^j,A_n^j]\}$
for any integer $r\ge 0$ (cf. the introduction).

In this section we shall derive the main results of the paper.
Firstly, we shall
give formulae for the dimensions of the K\"ulshammer
spaces $T_r(A_n^j)$, see Theorem \ref{mainresult} below.
Secondly, as an application we can distinguish certain of
the deformed preprojective algebras of type $L_n$ (over a
perfect field of characteristic $2$) up to
derived equivalence, see Theorem \ref{thm-derived} below.

The crucial link to distinguish algebras up to derived
equivalence by means of K\"ulshammer spaces has been
provided by the second author in \cite{Z}.
There it is shown that
for $K$ being a perfect field of characteristic $p>0$ and
for $\Lambda_1$ and $\Lambda_2$ being finite dimensional $K$-algebras
which are derived equivalent the codimensions of the K\"ulshammer
spaces are an invariant, i.e. for all $r\ge 0$ one has
\begin{equation} \label{formula-BHZ}
\dim_K \Lambda_1 - \dim_K T_r(\Lambda_1)=\dim_K \Lambda_2 -
\dim_K T_r(\Lambda_2).
\end{equation}

In \cite{LZZ} Liu, Zhou and the second author showed that
for any field $K$ of characteristic $p>0$ and any two finite dimensional
$K$-algebras $\Lambda_1$ and $\Lambda_2$, if $\Lambda_1$ and $\Lambda_2$
are stably equivalent of Morita type, then
\begin{equation}\label{LZZequation}
\dim_K T_r(\Lambda_1)-\dim_K [\Lambda_1,\Lambda_1] =
\dim_K T_r(\Lambda_2)-\dim_K [\Lambda_2,\Lambda_2]
\end{equation}
for all $r\ge 0$.


\subsection{Dimensions of K\"ulshammer spaces}
In this section we shall prove the main result on the dimensions of
the K\"ulshammer spaces $T_i(L_n^{X^{2j}})$ for the deformed preprojective
algebras. Before embarking on the general proof we shall give some
explicit examples which hopefully help the reader later by illustrating
the technicalities of the general arguments.
\medskip

\noindent
{\bf An example: the case $n=2$.}
Let us look at the algebras $A_2^j=L_2^{X^{2j}}$ as an illustration.
These algebras are given by a quiver with two vertices and relations
$\epsilon^4=0$, $\ol a_0 a_0=0$ and $\epsilon^2 + \epsilon^{2j+3}+a_0\ol a_0=0$.
Note that for $j\ge 1$ we get the undeformed algebra $A_2^1$ with relation
$\epsilon^2 + a_0\ol a_0=0$, whereas for $j=0$ we get a deformed
preprojective algebra $A_2^0$
with relation $\epsilon^2 + \epsilon^{3}+a_0\ol a_0=0$.

According to Proposition \ref{abasis} and Remark \ref{Rem-Cartan-dim}
the algebras $A_2^j$ are 10-dimensional
with a basis given by the paths
$$e_0,e_1,\epsilon, a_0,\ol a_0, a_0\ol a_0, \epsilon a_0,
\ol a_0\epsilon, a_0\ol a_0\epsilon, \ol a_0\epsilon a_0.$$
By Proposition \ref{explicitdescrofK} the commutator spaces $[A_2^j,A_2^j]$
are of dimension 6 and have a basis consisting of the elements
$$a_0,\ol a_0,a_0\ol a_0, \epsilon a_0, \ol a_0\epsilon,
a_0\ol a_0\epsilon - \ol a_0\epsilon a_0.$$
Note that all these bases are independent of $j$.

Now we consider the first K\"ulshammer space
$T_1(A_2^j) = \{x\in A_2^j\,|\,x^2\in [A_2^j,A_2^j]\}$.
For any $j\ge 0$ it is immediate from the relations that
the following seven basis elements of $A_2^j$ are contained
in the first K\"ulshammer space
$$\{a_0, \ol a_0, \epsilon a_0, \ol a_0\epsilon, a_0\ol a_0,
a_0\ol a_0\epsilon, \ol a_0\epsilon a_0\}\subset T_1(A_2^j).$$
On the other hand, it is a general observation that the
trivial paths $e_0,e_1$ can not be summands of an element in
a K\"ulshammer space (since trivial paths can't occur as summands
in an element from the commutator space).
This leave us with the remaining basis element $\epsilon$.
Here the situation changes for different $j$.

In the undeformed case $j\ge 1$ we have that
$\epsilon^2 = a_0\ol a_0 = [a_0,\ol a_0]\in [A_2^j,A_2^j]$
and hence $\epsilon\in T_1(A_2^j)$.

On the other, in the deformed case $j=0$ we have the
relation $\epsilon^2 = \epsilon^3 + a_0\ol a_0$ where
$a_0\ol a_0$ is a commutator but
$\epsilon^3=a_0\ol a_0\epsilon \not\in [A_2^0,A_2^0]$.
Therefore, $\epsilon\not\in T_1(A_2^0)$

In summary we have $\dim_K T_1(A_2^0)=7$ whereas
$\dim_K T_1(A_2^j)=8$ for all $j\ge 1$.

Using the result from \cite{Z} quoted above in (\ref{formula-BHZ})
we can deduce that the undeformed preprojective algebra
$A_2^1=L_2^{X^2}$ and the deformed preprojective algebra
$A_2^0=L_2^{X^0}$ are not derived equivalent.
Even in this small case $n=2$ this seems to be a nontrivial fact.


\medskip

\noindent
{\bf Another example: K\"ulshammer spaces for $n=3$.}
The algebras $A_3^j$ have dimension 28, and their commutator spaces
have dimension 22. There are many basis elements which are obviously
in each of the K\"ulshammer ideals $T_r(A_3^j)$, for $r\ge 1$, namely

$\bullet$ all non-closed paths in $\mathcal{B}$, giving 16 basis elements
(since they square to zero)

$\bullet$ closed paths of length $\ge 3$ (since the algebras have radical length
6 they also square to zero); so another six such basis elements are
$a_0\ol a_0\epsilon$, $\ol a_0\epsilon a_0$, $a_0a_1\ol a_1\ol a_0\epsilon$,
$a_1\ol a_1\ol a_0\epsilon a_0$, $\ol a_1\ol a_0\epsilon a_0a_1$
$a_0a_1\ol a_1\ol a_0$.

$\bullet$ the two basis elements $a_0\ol a_0$ and $a_1\ol a_1$
(since $(a_1\ol a_1)^2=0$ and
$(a_0\ol a_0)^2$ is in the commutator space by the proof of Lemma 9).

Hence, $\dim_K T_r(A_3^j) \ge 24$ for all $r\ge 1$ and all $j$.

Given that the three trivial paths are not involved in any element
of the K\"ulshammer space, there is only one remaining basis element
to consider, namely $\epsilon$.

We start with the first K\"ulshammer space $T_1(A_3^j)$.

For $j=2$ we have $\epsilon^2 = a_0\ol a_0\in [A_3^2,A_3^2]$, i.e.
$\epsilon\in T_1(A_3^j)$.

For $j=1$ we get the following congruences modulo the commutator space
$$\epsilon^2 = \epsilon^5 +a_0\ol a_0 \equiv \epsilon^5
= \epsilon^8 +a_0\ol a_0 \epsilon^3 = a_0\ol a_0(\epsilon^6+a_0\ol a_0\epsilon)
= a_0\ol a_0a_0\ol a_0\epsilon \not\equiv 0$$
i.e. $\epsilon\not\in T_1(A_3^1)$.

For $j=0$ we similarly get the following congruences modulo the commutator space
$$\epsilon^2 = \epsilon^3 +a_0\ol a_0 \equiv \epsilon^3
= \epsilon^4 +a_0\ol a_0 \epsilon = \epsilon^5+a_0\ol a_0\epsilon^2+a_0\ol a_0\epsilon
= \epsilon^6 + a_0\ol a_0\epsilon^3+a_0\ol a_0\epsilon^2+a_0\ol a_0\epsilon
$$
$$ = a_0\ol a_0\epsilon^3 + a_0\ol a_0(\epsilon^3+a_0\ol a_0)+a_0\ol a_0\epsilon
\equiv a_0\ol a_0\epsilon \not\equiv 0$$
i.e. $\epsilon\not\in T_1(A_3^0)$.
Altogether we get
$\dim_K T_1(A_3^j) = \left\{
\begin{array}{ll}
24 & if~j=0,1 \\ 25 & if~j=2
\end{array}
\right.
$

Now we consider the second K\"ulshammer space
$T_2(A_3^j)=\{x\in A_3^j\,|\,x^4\in [A_3^j,A_3^j]\}$.
Again it only remains to consider the basis element $\epsilon$.

For $j=2$ there is nothing to check since $T_2(A_3^2)$ already attained
the maximal possible dimension 25.

For $j=1$ we get the following congruences (modulo commutator space)
$$\epsilon^4 = \epsilon^7 + a_0\ol a_0\epsilon^2 = a_0\ol a_0\epsilon^5
+ a_0\ol a_0a_0\ol a_0 \equiv 0$$
i.e. $\epsilon\in T_2(A_3^1)$.

Similarly we get for $j=0$ (modulo commutator space)
$$\epsilon^4 = \epsilon^5 + a_0\ol a_0\epsilon^2 =
\epsilon^6 + a_0\ol a_0\epsilon^3 + a_0\ol a_0\epsilon^3 + a_0\ol a_0
a_0\ol a_0 \equiv 0$$
i.e. $\epsilon\in T_2(A_3^0)$.

Altogether we get
$\dim_K T_2(A_3^j) = 25$ for all $j$.
\medskip

We now formulate the main result of this section.

\begin{Theorem} \label{mainresult} Let $K$ be a perfect field of characteristic $2$.
Then for all $0\le j < n$ we have
\begin{eqnarray*}
\dim_K T_i(L_n^{X^{2j}}) - \dim_K [L_n^{X^{2j}},L_n^{X^{2j}}]
&=&
n-\max\left(\left\lceil\frac{2n-(2^{i+1}-2)j-(2^{i+1}-1)}{2^{i+1}}\right\rceil,\;\;0\right).
\end{eqnarray*}
\end{Theorem}

\begin{proof}
Lemma~\ref{basisofcommutatorquotient} provided a set of coset generators of
the commutator space $[L_n^{X^{2j}},L_n^{X^{2j}}]$ in $L_n^{X^{2j}}$,
namely $\{e_0,\ldots,e_{n-1}\}\cup \{\epsilon^{(2k+1)}|\;0\le k\le n-1\}$.
For our purpose of determining the K\"ulshammer ideals we can discard the
trivial paths since they can never be involved in any element of
a K\"ulshammer ideal.
Therefore in order to compute $T_i(L_n^{X^{2j}})$ we need to see when a $2^i$-th
power of a
linear combination of elements $\{\epsilon^{(2k+1)}|\;0\le k\le n-1\}$ lies in
the commutator space $[L_n^{X^{2j}},L_n^{X^{2j}}]$.

Note that $\epsilon\in e_0L_n^{X^{2j}}e_0$, and that by Proposition
\ref{explicitdescrofK} a basis of the intersection
$e_0L_n^{X^{2j}}e_0 \cap [L_n^{X^{2j}},L_n^{X^{2j}}]$
is given by the paths
$a_0a_1\ldots a_{\ell}\ol a_{\ell}\ol a_{\ell-1}\ldots \ol a_0$
where $0\le \ell\le n-2$. Moreover, we have that
$$a_0a_1\ldots a_{\ell}\ol a_{\ell}\ol a_{\ell-1}\ldots \ol a_0
= (a_0\ol a_0)^{\ell+1} =(\epsilon^2+\epsilon^{2j+3})^{\ell+1}
$$
(no signs occurring since we are in characteristic 2).
This means that in order to obtain the desired formula for
the dimension of $T_i(L_n^{X^{2j}})$
one needs to consider the $K$-vector space
$$
\tilde T_{n,j}(i):=\left\{\sum_{k=0}^{n-1}b_k\epsilon^{2k+1}
\left|\;\left(\sum_{k=0}^{n-1}b_k\epsilon^{2k+1}\right)^{2^i}
\in\left<(\epsilon^2+\epsilon^{2j+3})^m|\;1\le m\le n-1\right>_K\right.\right\}
$$
whose dimension is equal to the dimension of the factor space
$T_i(L_n^{X^{2j}})/[L_n^{X^{2j}},L_n^{X^{2j}}]$.

In order to determine this dimension we therefore have to express an element
$\left(\sum_{k=0}^{n-1}b_k\epsilon^{2k+1}\right)^{2^i}$
as a linear combination of the form
$\sum_{m=1}^{n-1}c_m(\epsilon^2+\epsilon^{2j+3})^m$
for $c_m\in K$.

\bigskip

\paragraph{\bf First Step}
%
We shall
reduce the problem to the case of $K$ being the prime field of characteristic $2$.

As is described in the remarks preceding the statement of the theorem we have to
give the dimension of the $K$-vector space
$$
\left\{\sum_{k=0}^{n-1}b_k\epsilon^{2k+1}\in K[\epsilon]\left|\;\left(\sum_{k=0}^{n-1}b_k\epsilon^{2k+1}\right)^{2^i}
\in\left<(\epsilon^2+\epsilon^{2j+3})^m|\;m\in\N\right>_K\right.\right\}
$$
Let $$U:=\left<(\epsilon^2+\epsilon^{2j+3})^m|\;m\in\N\right>_K\subseteq K[\epsilon]/\epsilon^{2n}.$$
Then let $$V:=\left<\epsilon^{2k+1}\right>_K\subseteq K[\epsilon]/\epsilon^{2n}$$
and let $\mu:K[\epsilon]/\epsilon^{2n}\lra K[\epsilon]/\epsilon^{2n}$ given by
$\mu(x):=x^2$. Then
$$
\left\{\sum_{k=0}^{n-1}b_k\epsilon^{2k+1}\in K[\epsilon]\left|\;\left(\sum_{k=0}^{n-1}b_k\epsilon^{2k+1}\right)^{2^i}
\in\left<(\epsilon^2+\epsilon^{2j+3})^m|\;m\in\N\right>_K\right.\right\}
=V\cap \left(\mu^i\right)^{-1}(U).
$$
Now, $U=U_0\otimes_{\F_2}K$ and $V:=V_0\otimes_{\F_2}K$ for
$U_0$ and $V_0$ being defined as $U$ and $V$, but with $\F_2$ as base field.
If $K$ is perfect, then
$$
V\cap \left(\mu^i\right)^{-1}(U)=(V_0\otimes_{\F_2}K)\cap \left(\mu^i\right)^{-1}(U_0\otimes_{\F_2}K)=\left(V_0\cap
\left(\mu^i\right)^{-1}(U_0)\right)\otimes_{\F_2}K
$$

Hence the dimension of the vector space
can be computed in $\F_2$. We hence may assume that $K=\F_2$.

Since $K$ is assumed to be the prime field, we get $b^2=b$ for all $b\in K$, and so
we need to find coefficients $c_m\in K$ so that
$$\left(\sum_{k=0}^{n-1}b_k\epsilon^{2^i(2k+1)}\right)=
\left(\sum_{k=0}^{n-1}b_k\epsilon^{(2k+1)}\right)^{2^i}=
\sum_{m=1}^{n-1}c_m\left(\epsilon^2+\epsilon^{2j+3}\right)^m.$$

\bigskip

\paragraph{\bf Second step}
In the course of the proof we shall need to know whether certain
binomial coefficients are even or odd. More precisely, write a natural
number as $2^a v$ with $v$ odd and $a\in\mathbb{N}\cup\{0\}$. Then we have that
$$\left\{ u\in\N\setminus\{0\}\;\left|\;{{2^a v}\choose u}
\mbox{~odd}\right.\right\}\subseteq 2^a\Z\
\mbox{~\hskip2.5cm~~and~~\hskip2.5cm~}
\min\left\{u\in\N\setminus\{0\}\;\left|\;{{2^av}\choose u}\mbox{~odd}\right.\right\}=2^a.$$
In fact, both statements follow easily from the following well-known result on
binomial coefficients, going back to Lucas \cite{Lucas}; for a
proof see e.g. \cite{Fine}:
{\em Let $p$ be a prime, and let natural numbers
$M=\sum M_ip^i$ and $N=\sum N_ip^i$ be given in their $p$-adic expansion.
Then
${M\choose N} \equiv \prod_i {M_i\choose N_i} \mod p.$
}

\medskip

We remark further that binomial coefficients are integers. Hence, seen in $K$
they actually belong to the prime field. If $K$ is of characteristic $2$, then
a binomial coefficient can only have values $0$ or $1$.

\bigskip

\paragraph{\bf Third step.}
We need to study for which $b_0,\ldots,b_{n-1}$ given, there
exist coefficients $c_m\in K$ so that
\begin{equation} \label{3rd-eqn1}
\sum_{k=0}^{n-1}b_k\epsilon^{2^{i}(2k+1)}=
\sum_{m=1}^{n-1}c_m(\epsilon^2+\epsilon^{2j+3})^m.
\end{equation}

We first determine a lower bound for the indices of the non-vanishing
coefficients $b_k$.
Denote by $k_0$ the smallest integer $k$ so that $b_k\neq 0$.
Then formula (\ref{3rd-eqn1}) reads
\begin{equation} \label{3rd-eqn2}
\sum_{k=k_0}^{n-1}b_k\epsilon^{2^i(2k+1)}=
\sum_{m=1}^{n-1}c_m(\epsilon^2+\epsilon^{2j+3})^m.
\end{equation}
Comparing the smallest powers of $\epsilon$ occurring on either side
of equation (\ref{3rd-eqn2}) we can deduce that
$c_{m}=0$ for $m<2^{i-1}(2k_0+1)$ and $c_{2^{i-1}(2k_0+1)}\neq 0$.
Hence equation (\ref{3rd-eqn2}) now reads
\begin{equation}\label{3rd-eqn3}
\sum_{k=k_0}^{n-1}b_k\epsilon^{2^i(2k+1)}=
\sum_{m=2^{i-1}(2k_0+1)}^{n-1}c_m(\epsilon^2+\epsilon^{2j+3})^m.
\end{equation}
Using the statements
on the parity of binomial coefficients from the second step
and the
fact that the base field is of characteristic $2$ we have that
\begin{eqnarray*}
(\epsilon^2+\epsilon^{2j+3})^{2^{i-1}(2k_0+1)}
&=&\epsilon^{2^i(2k_0+1)}+{{2^{i-1}(2k_0+1)}\choose 2^{i-1}}\cdot
\left(\epsilon^2\right)^{2^{i-1}(2k_0+1)-2^{i-1}}\cdot\left(\epsilon^{2j+3}\right)^{2^{i-1}}+\\
&&+\mbox{ higher powers of $\epsilon^{2^{i-1}}$}\\
&=&\epsilon^{2^i(2k_0+1)}+
\epsilon^{2^{i-1}(4k_0+2j+3)}+\mbox{ higher powers of $\epsilon^{2^{i-1}}$}.
\end{eqnarray*}
Hence as long as $2^{i-1}(4k_0+2j+3)< 2n$ (i.e. $\epsilon^{2^{i-1}(4k_0+2j+3)}$
does not vanish), a non-zero scalar multiple of
$\epsilon^{2^{i-1}(4k_0+2j+3)}$ occurs on the right hand
side of equation (\ref{3rd-eqn3}). However, it can not occur on the left hand side
of equation (\ref{3rd-eqn3}) since $2^{i-1}(4k_0+2j+3)$ is not divisible by $2^i$.
So $\epsilon^{2^{i-1}(4k_0+2j+3)}$ would also have to be a
term of some other summand in
$\sum_{m=2^{i-1}(2k_0+1)}^{n-1}c_m(\epsilon^2+\epsilon^{2j+3})^m$
(so that the terms can cancel out).

For $i=1$ this is impossible, since for $m>2^{i-1}(2k_0+1)$ the smallest
possible odd exponent in $(\epsilon^2+\epsilon^{2j+3})^m$
is already larger than $2^{i-1}(4k_0+2j+3)$.
Hence for $i=1$ we must have that
$4k_0+2j+3\geq 2n$ which implies that
$k_0=\max\left(\left\lceil\frac{2n-2j-3}{4}\right\rceil,0\right).$
Note that we indeed have to take the maximum with $0$ here since
the index $k_0$ is non-negative by definition.
\smallskip

Suppose now that $i\ge 2$. Then we claim that the only
possibility to cancel the above term $\epsilon^{2^{i-1}(4k_0+2j+3)}$
is to put
$$c_{2^{i-2}(4k_0+2 j+3)}= c_{2^{i-1}(2k_0+1)}\neq 0.$$
In fact, on the one hand we have that
$$(\epsilon^2+\epsilon^{2j+3})^{2^{i-2}(4k_0+2 j+3)} =
\epsilon^{2^{i-1}(4k_0+2 j+3)}+\mbox{ higher powers of $\epsilon^{2^{i-2}} $}$$
so that the desired term cancels; on the other hand, it could not cancel
for a smaller index $m$ since this would have to satisfy
$m\ge 2^{i-1}(2k_0+3)$ (note that the exponents on the left hand side
of equation (\ref{3rd-eqn1}) are divisible by $2^i$) and then by Lucas' theorem
above (cf. second step) the second term
in $(\epsilon^2+\epsilon^{2j+3})^m$ already has exponent
$$2^i(2k_0+2)+(2j+3)2^{i-1} = 2^{i-1}(4k_0+2j+7) > 2^{i-1}(4k_0+2j+3).$$

In a similar way, again using the second step
and that the base field is of characteristic $2$,
we further get
\begin{eqnarray*}
(\epsilon^2+\epsilon^{2j+3})^{2^{i-2}(4k_0+2j+3)}&=&\epsilon^{2^{i-1}(4k_0+2j+3)}+
\left(\epsilon^2\right)^{2^{i-2}(4k_0+2j+3)-2^{i-2}}\cdot\left(\epsilon^{2j+3}\right)^{2^{i-2}}+\\
&&+\mbox{ higher powers of $\epsilon^{2^{i-2}}$}\\
&=&\epsilon^{2^{i-1}(4k_0+2j+3)}+\epsilon^{2^{i-2}(8k_0+6j+7)}+\mbox{ higher powers of $\epsilon^{2^{i-2}}$}.
\end{eqnarray*}
Completely analogous to the case $i=1$ above we can deduce
that for $i=2$ we have
$8k_0+6j+7\geq 2n$ and therefore
$k_0=\max\left(\left\lceil\frac{2n-6j-7}{8}\right\rceil,0\right)$
in the case $i=2$. This is the second correction step.

\smallskip

We shall show by induction on $s$, that the lowest power of $\epsilon$ appearing in the
sum on the right hand side of equation (\ref{3rd-eqn3})
after $s$ corrections is
$$\epsilon^{2^{i-s}(2^sk_0+(2^s-2)\cdot j+(2^s-1))}.$$

The cases $s\in\{1,2\}$ have been treated above. Suppose the formula is shown
for some $s<i$. Then we shall show the formula for $s+1$:
We shall need to correct with $c_{2^{i-s-1}(2^sk_0+(2^s-2)j+(2^s-1))}\neq 0$ and get
higher error terms as follows:
\begin{eqnarray*}
(\epsilon^2+\epsilon^{2j+3})^{2^{i-s-1}(2^sk_0+(2^s-2)j+(2^s-1))}&=&
\epsilon^{2^{i-s}(2^sk_0+(2^s-2)j+(2^s-1))}+\\
&&+\left(\epsilon^2\right)^{2^{i-s-1}(2^sk_0+(2^s-2)j+(2^s-1))-2^{i-s-1}}
\cdot\left(\epsilon^{2j+3}\right)^{2^{i-s-1}}+\\
&&+\mbox{ higher powers of $\epsilon^{2^{i-s-1}}$}\\
&=&\epsilon^{2^{i-s}(2^sk_0+(2^s-2)j+(2^s-1))}+\\
&&+\epsilon^{2^{i-s-1}(2^{s+1}k_0+2\cdot(2^s-2)j+2\cdot (2^s-1-1))+2^{i-s-1}(2j+3)}+\\
&&+\mbox{ higher powers of $\epsilon^{2^{i-s-1}}$}\\
&=&\epsilon^{2^{i-s}(2^sk_0+(2^s-2)j+(2^s-1))}+\\
&&+\epsilon^{2^{i-s-1}(2^{s+1}k_0+(2^{s+1}-4+2)j+ (2^{s+1}-4+3))}+\\
&&+\mbox{ higher powers of $\epsilon^{2^{i-s-1}}$}\\
&=&\epsilon^{2^{i-s}(2^sk_0+(2^s-2)j+(2^s-1))}+\\
&&+\epsilon^{2^{i-s-1}(2^{s+1}k_0+(2^{s+1}-2)j+ (2^{s+1}-1))}+\\
&&+\mbox{ higher powers of $\epsilon^{2^{i-s-1}}$}\\
\end{eqnarray*}
which shows the formula for $s+1$.

Hence, we may correct the error terms by successively choosing
appropriate $c_m$ for higher and higher $m$, as long as $s<i$. If $s=i$
then the error term cannot be annihilated, and therefore it must be $0$.
Therefore  $$2^{i+1}k_0+(2^{i+1}-2)\cdot j+(2^{i+1}-1)\geq 2n$$
which means $$k_0\geq \frac{2n-(2^{i+1}-2)j-(2^{i+1}-1)}{2^{i+1}}$$ and therefore
$$k_0=\max\left(\left\lceil\frac{2n-(2^{i+1}-2)j-(2^{i+1}-1)}{2^{i+1}}
\right\rceil,\;\;0\right).$$
\smallskip

\noindent
{\bf Fourth step.}
Suppose
$$k\geq k_0=\max\left(\left\lceil\frac{2n-(2^{i+1}-2)j-(2^{i+1}-1)}{2^{i+1}}\right\rceil\right).$$
We shall prove that then
$$\epsilon^{2^i(2k+1)}\in\left<(\epsilon^2+\epsilon^{2j+3})^m|\;m\in\N\right>_K.$$
To this end we put $c_{2^{i-1}(2k+1)}=1$
and get by the second step that
\begin{eqnarray*}
\left(\epsilon^2+\epsilon^{2j+3}\right)^{2^{i-1}(2k+1)}-\epsilon^{2^i(2k+1)}&=&
\left(\epsilon^2\right)^{2^{i-1}(2k+1)-2^{i-1}}
\cdot\left(\epsilon^{2j+3}\right)^{2^{i-1}}+
\mbox{ higher order powers of $\epsilon^{2^{i-1}}$}\\
&=&\epsilon^{2^{i-1}(4k+2j+3)}+\mbox{ higher order powers of $\epsilon^{2^{i-1}}$}
\end{eqnarray*}

Hence, we can choose coefficients $c_{2^{i-1}m}$ for certain $m$
so that
$$\epsilon^{2^i(2k+1)}-\sum c_{2^{i-1}m}(\epsilon^2+\epsilon^{2j+3})^{2^{i-1}m}$$
is a direct sum of terms $\epsilon^{2^{i-1}\ell} $ where $\ell\geq 4k+2j+3$.

If $i=1$, we are done since then $\epsilon^{2^{i-1}\ell} =0$ since $\ell$ was
chosen in a way that
$$\ell\geq 4k+2j+3\geq 4k_0+2j+3\geq 2n.$$

If $i\geq 2$, put $c_{2^{i-2}\ell}=1$
for all terms $\epsilon^{2^{i-1}\ell}$ of the powers of $\epsilon^{2^{i-1}}$
occurring in the above difference
$$\epsilon^{2^i(2k+1)}-\sum c_{2^{i-1}m}(\epsilon^2+\epsilon^{2j+3})^{2^{i-1}m}.$$
We know that each of these $\ell$ satisfies $\ell\geq 4k+2j+3$, so that
all the coefficients  $2^{i-2}\ell$ are bigger than $2^{i-2}(4k_0+2j+3)$.

We compute
\begin{eqnarray*}
\left(\epsilon^2+\epsilon^{2j+3}\right)^{2^{i-2}\ell}&=&
\epsilon^{2^{i-1}\ell}+{{2^{i-2}\ell}\choose {2^{i-2}}}
\left(\epsilon^2\right)^{2^{i-2}\ell-2^{i-2}}
\cdot\left(\epsilon^{2j+3}\right)^{2^{i-2}}+\\
&&+\mbox{ higher order powers of $\epsilon^{2^{i-2}}$}\\
&=&\epsilon^{2^{i-1}\ell}+
{{2^{i-2}\ell}\choose {2^{i-2}}}\epsilon^{2^{i-2}(2\ell+2j+1)}+
\mbox{ higher order powers of $\epsilon^{2^{i-2}}$}
\end{eqnarray*}
Again, if $i=2$ we are done since $$2\ell+2j+1\geq 2(4k_0+2j+3)+2j+1=8k_0+6j+7\geq 2n$$
and hence
$\epsilon^{2^{i-2}(2\ell+2j+1)}=0$ for all $\ell$
which may occur by definition of $k_0$.

We use induction on $s$ on the statement that we may choose
$c_{2^{i-s}m}$ so that only powers $\epsilon^{2^{i-s}\ell}$ occur
in the difference
$$\epsilon^{2^i(2k+1)}-\sum_{m=1}^{n-1}c_{2^{i-s}m}(\epsilon^2+\epsilon^{2j+3})^{2^{i-s}m}$$
with $\ell\geq (2^sk_0+(2^s-2)\cdot j+(2^s-1)).$

The statement is true for $s=1$ and $s=2$ by the above discussion.
Suppose it is true for $s\leq i$. We shall prove it for $s+1$.

For every term $\epsilon^{2^{i-s}\ell}$ which occurs as a summand in
$$\epsilon^{2^i(2k+1)}-\sum_{m=1}^{n-1}c_{2^{i-s}m}(\epsilon^2+\epsilon^{2j+3})^{2^{i-s}m}$$
we put $c_{2^{i-s-1}\ell}=1$ and then
we compute
\begin{eqnarray*}
\left(\epsilon^2+\epsilon^{2j+3}\right)^{2^{i-s-1}\ell}&=&
\epsilon^{2^{i-s}\ell}+{{2^{i-s-1}\ell}\choose {2^{i-s-1}}}
\left(\epsilon^2\right)^{2^{i-s-1}\ell-2^{i-s-1}}
\cdot\left(\epsilon^{2j+3}\right)^{2^{i-s-1}}+\\
&&+\mbox{ higher order powers of $\epsilon^{2^{i-s-1}}$}\\
&=&\epsilon^{2^{i-s}\ell}+
{{2^{i-s-1}\ell}\choose {2^{i-s-1}}}\epsilon^{2^{i-s-1}(2\ell+2j+1)}+
\mbox{ higher order powers of $\epsilon^{2^{i-2}}$}
\end{eqnarray*}

Now, by induction hypothesis $\ell\geq (2^sk_0+(2^s-2)\cdot j+(2^s-1)).$
Hence
$$2\ell+2j+1\geq 2\cdot (2^sk_0+(2^s-2)\cdot j+(2^s-1))+2j+1=
2^{s+1}k_0+(2^{s+1}-2)j+(2^{s+1}-1)$$
which is the statement for $s+1$.

But now, finally for $s=i$ we get that the error terms are $\epsilon^t$ where
$$t\geq 2^{i+1}k_0+(2^{i+1}-2)j+(2^{i+1}-1)\geq 2n$$ by definition of $k_0$
and hence the error terms are $0$.

Therefore,
$$
\epsilon^{2^i(2k+1)}\in\left\{\sum_{k=0}^{n-1}b_k\epsilon^{2k+1}\in K[\epsilon]\left|\;\left(\sum_{k=0}^{n-1}b_k\epsilon^{2k+1}\right)^{2^i}
\in\left<(\epsilon^2+\epsilon^{2j+3})^m|\;m\in\N\right>_K\right.\right\}\;\;
\mbox{ for all }k\geq k_0.
$$

\bigskip

\paragraph{\bf Fifth step}

Now we are able to compute the dimension of
$$\tilde T_{n,j}(i)=\left\{\sum_{k=0}^{n-1}b_k\epsilon^{2k+1}\in K[\epsilon]\left|\;\;\left(\sum_{k=0}^{n-1}b_k\epsilon^{2k+1}\right)^{2^i}
\in\left<(\epsilon^2+\epsilon^{2j+3})^m|\;m\in\N\right>_K\right.\right\}
$$
We know by the third and fourth step that
$$\tilde T_{n,j}(i)=\left\{\sum_{k=0}^{n-1}b_k\epsilon^{2k+1}\in K[\epsilon]\left|\;
b_k=0\mbox{ for }k<\left\lceil\frac{2n-(2^{i+1}-2)j-(2^{i+1}-1)}{2^{i+1}}\right\rceil
\right.\right\}$$
The dimension of this space is therefore
$$
\dim \left(T_i(L_n^{X^{2j}})/[L_n^{X^{2j}},L_n^{X^{2j}}]\right)=
\dim\tilde T_{n,j}(i)=n-\left\lceil\frac{2n-(2^{i+1}-2)j-(2^{i+1}-1)}{2^{i+1}}\right\rceil
$$

\bigskip

This finishes the proof.
\end{proof}

\subsection{Consequences for derived equivalence and
stable equivalence of Morita type}
In this section we shall address the main motivational question for
this paper, namely when deformed preprojective algebras of type $L$
are derived equivalent or stably equivalent of Morita type.
As main application of our results on K\"ulshammer spaces we
can obtain partial answers to these problems.

For both notions of equivalence it is in general a difficult question
to decide whether two algebras are equivalent or not.

According to \cite{Bialkowski-ICRA-Abstract}, 
for the deformed preprojective algebras of type
$L_n$, Bia\l kowski, Erdmann and Skowro\'nski
are going to show in \cite{BESIII} that a for an algebraically closed field $K$
the set of algebras
$\{L_n^{X^{2j}}\,|\,0\le j\le n-1\}$
gives a complete set of representatives for the Morita equivalence
classes.

As an application of our result on K\"ulshammer spaces we can now distinguish
several of these algebras up to derived equivalence, and up to stable
equivalence of Morita type.

\begin{Theorem} \label{thm-derived}
Let $K$ be a perfect field of characteristic 2.
\begin{enumerate}
\item[{(a)}]
If two deformed preprojective
algebras $L_n^{p}$ and $L_m^{q}$ are stably equivalent of Morita type
or derived equivalent,
then $n=m$.
\item[{(b)}]
For $n\in\mathbb{N}$ let $j,k\in\{0,1\ldots,n-1\}$ be different numbers
such that
$\{j,k\}\neq \{n-2r,n-2r-1\}$ for every $1\le r\le \lceil \frac{n-2}{2}\rceil$.
Then the deformed preprojective algebras
$L_n^{X^{2j}}$ and $L_n^{X^{2k}}$
are not stably equivalent of Morita type, and also not
derived equivalent.
\end{enumerate}
\end{Theorem}

\begin{proof}
(a) It is well-known that the number of simple modules is a derived invariant.

Moreover, by a result of C. Xi \cite[Proposition 5.1]{Xi},
the absolute value of the determinant of the Cartan matrix of an algebra
is invariant under stable equivalence of Morita type. For the deformed
preprojective algebras $L_n^p$ the Cartan determinant is $2^n$,
see Remark~\ref{Rem-Cartan-dim}, so the result follows.
\smallskip

(b) We use the first K\"ulshammer space or more precisely
the following difference occurring in
Theorem \ref{mainresult} for the case $i=1$,
\begin{eqnarray} \label{eqn-T1}
\dim_K\left(T_1(L_n^{X^{2j}})\right)-\dim_K\left([L_n^{X^{2j}},L_n^{X^{2j}}]
\right) &=&
n-\max\left(\left\lceil\frac{2n-2j-3}{4}
\right\rceil,0\right).
\end{eqnarray}
By a result of Liu, Zhou and the
second author \cite[Corollary 7.5]{LZZ} this number is invariant
under stable equivalences of Morita type.
Since the numerator $2n-2j-3$
is congruent to 1 or 3 modulo 4 we have
$$\left\lceil \frac{2n-2j-3}{4}\right\rceil = \left\lceil \frac{2n-2j-2}{4}
\right \rceil
= \left\lceil \frac{n-j-1}{2}\right\rceil.
$$
Note that for all the values $j\in\{0,\ldots,n-1\}$ this number is
non-negative, so that equation (\ref{eqn-T1}) reads
\begin{eqnarray*}
\dim_K\left(T_1(L_n^{X^{2j}})\right)-\dim_K\left([L_n^{X^{2j}},L_n^{X^{2j}}]
\right) &=&
n-\left\lceil\frac{n-j-1}{2}
\right\rceil.
\end{eqnarray*}
For fixed $n$, this invariant becomes equal for two different
values $j,k\in\{0,\ldots,n-1\}$
precisely when $\{j,k\}=\{n-2r,n-2r-1\}$ for some
$1\le r\le \lceil \frac{n-2}{2}\rceil$. This proves the assertion on stable
equivalence of Morita type.

The statement on derived equivalence follows immediately by
using a result by Rickard \cite{Rickard} and Keller and Vossieck
\cite{KV}
saying that for selfinjective algebras (recall that our algebras $L_n^p$
are even symmetric by Theorem \ref{thm-symmetric}) any derived equivalence induces
a stable equivalence of Morita type.
\end{proof}

\begin{Remark}
{\em
(1) In the above theorem we have for simplicity only exploited the first K\"ulshammer
space, but of course one could also use higher K\"ulshammer spaces for
distinguishing algebras up to derived equivalence, or up stable equivalence
of Morita type. For explicit examples of deformed preprojective algebras
of type $L$ see Example \ref{example-Ln} below.
\smallskip

(2) Note that part (b) of the above theorem in particular applies
whenever $|j-k|\geq 2$.
\smallskip

(3)
For any $n\in\mathbb{N}$ the (undeformed) preprojective algebra
$L_n^{X^{2(n-1)}}$ is not stably equivalent of Morita type, and also not derived
equivalent, to any of the
algebras $L_n^{X^{2j}}$ for $j\in\{0,\ldots,n-2\}$. In fact,
by the preceding remark it suffices to distinguish the algebras
$L_n^{X^{2(n-1)}}$ and $L_n^{X^{2(n-2)}}$; but $j=n-1$ and $k=n-2$
are not of the form
$\{n-2r,n-2r-1\}$ for some $1\le r\le \lceil \frac{n-2}{2}\rceil$.
}
\end{Remark}

\begin{Example} \label{example-Ln}
{\em
\begin{enumerate}
\item[{(1)}] {\bf The case $n=2$ revisited.} Up to Morita equivalence there are
two deformed preprojective algebras of type $L_2$, namely $L_2^{X^0}$ and
$L_2^{X^2}$. They are not stably equivalent of Morita type
(and hence not derived equivalent) by Theorem \ref{thm-derived}\,(5).
So we have a complete classification
of deformed preprojective algebras of type $L_2$,
up to stable equivalence of Morita type (and up to derived equivalence)
\item[{(2)}] {\bf The case $n=3$ revisited.} There are three deformed preprojective
algebras of type $L_3$, namely $L_2^{X^0}$, $L_2^{X^2}$ and
$L_2^{X^4}$. The algebra $L_2^{X^4}$ is not stably equivalent
of Morita type (and hence not derived equivalent)
to the other two algebras.

But with the K\"ulshammer spaces we can not distinguish the algebras
$L_2^{X^0}$ and $L_2^{X^2}$. We don't know whether these are stably
equivalent of Morita type (or derived equivalent), or not.
\item[{(3)}] {\bf The case $n=5$.} For the five algebras (up to Morita equivalence)
$L_2^{X^{2j}}$ where $j\in\{0,1,2,3,4\}$ we get the following numbers for
the differences
\begin{eqnarray*}
\dim_K\left(T_i(L_5^{X^{2j}})\right)-\dim_K\left([L_5^{X^{2j}},L_5^{X^{2j}}]
\right) &=&
5-\max\left(\left\lceil\frac{10-(2^{i+1}-2)j-(2^{i+1}-1)}{2^{i+1}}
\right\rceil,0\right)
\end{eqnarray*}
which are invariants under derived equivalence and under stable equivalence
of Morita type.
\begin{center}
\begin{tabular}{|c||c|c|c|c|c|}
\hline
$i\setminus j$ & 0 & 1 & 2 & 3 & 4 \\
\hline\hline
1 & 3 & 3 & 4 & 4 & 5 \\
\hline
2 & 4 & 5 & 5 & 5 & 5 \\
\hline
$\ge 3$ & 5 & 5 & 5 & 5 & 5 \\
\hline
\end{tabular}
\end{center}
Therefore the algebras $L_5^{X^{0}}$, $L_5^{X^{2}}$, $L_5^{X^{4}}$ and
$L_5^{X^{8}}$ are pairwise not stably equivalent of Morita type (and hence
pairwise not derived equivalent). Note that
$L_5^{X^{0}}$ and $L_5^{X^{2}}$ can only be distinguished by the second
K\"ulshammer spaces.

It remains open
whether $L_5^{X^{4}}$ and $L_5^{X^{6}}$ are stably equivalent of Morita
type (or derived equivalent),
or not.
\end{enumerate}
}
\end{Example}

\end{document}